\documentclass[11pt,reqno]{amsart}
\usepackage{amsmath}
\usepackage{amssymb}
\usepackage{amsthm}
\usepackage{color}
\usepackage{eucal}
\usepackage{tikz}
\usepackage{gastex}
\usetikzlibrary{decorations.pathmorphing}
\DeclareSymbolFont{rsfscript}{OMS}{rsfs}{m}{n}
\DeclareSymbolFontAlphabet{\mathrsfs}{rsfscript}
\def\softd{{\leavevmode\setbox1=\hbox{d}\hbox
            to 1.15\wd1{d\kern-0.2ex{\char039}\hss}}}    
\def\softl{l\kern-0.3ex\raise0.1ex\hbox{'}\kern-0.3ex}   
\numberwithin{equation}{section}
\newtheorem{Thm}{Theorem}[section]
\newtheorem{Prop}[Thm]{Proposition}
\newtheorem{Lemma}[Thm]{Lemma}

\newtheorem{Cor}[Thm]{Corollary}

\DeclareSymbolFont{rsfscript}{OMS}{rsfs}{m}{n}
\DeclareSymbolFontAlphabet{\mathrsfs}{rsfscript}

\theoremstyle{remark}
\newtheorem{Rmk}{Remark}
\newtheorem{Problem}{Problem}

\def\inv{^{-1}}
\def\wh#1{\widehat{#1}}
\def\til#1{\widetilde {#1}}
\def\ol#1{\overline{#1}}
\def\rb{{-\!\!\!-\!\!\!-\!\!\!-\!\!\!\!\!\rightarrow}}
\def\rbl{{-\!\!\!-\!\!\!-\!\!\!-\!\!\!-\!\!\!\!\!\rightarrow}}
\def\Ga{\Gamma}
\def\ga{\gamma}

\def\De{\Delta}
\def\Th{\Theta}
\def\fF{\mathfrak{F}}
\def\fS{\mathfrak{S}}
\def\vp{\varphi}
\def\vpl{\varprojlim}
\def\cA{\mathrsfs A}

\def\cH{\mathrsfs H}
\def\cK{\mathrsfs K}
\def\cR{\mathrsfs R}

\def\ka{\kappa}

\def\ga{\gamma}

\begin{document}

\title{The geometry of profinite graphs revisited}
\author{K.~Auinger}
\address{Fakult\"at f\"ur Mathematik, Universit\"at Wien,
Oskar-Morgenstern-Platz 1, A-1090 Wien, Austria}
\email{karl.auinger@univie.ac.at}

\begin{abstract}
For a formation $\mathfrak{F}$ of finite groups, tight connections are established between the pro-$\mathfrak{F}$-topology of a finitely generated free group $F$ and the geometry of the Cayley graph $\Ga(\wh{F_{\fF}})$ of the pro-$\mathfrak{F}$-completion $\wh{F_{\fF}}$ of $F$. For example, the Ribes--Zalesski{\u\i}-Theorem is proved for the pro-$\mathfrak{F}$-topology of $F$ in case $\Ga(\wh{F_{\fF}})$ is a tree-like graph. All these results are established by purely geometric proofs, more directly and more transparently than in earlier papers, without the use of inverse monoids.  Due to the richer structure provided by formations (compared to varieties), new examples of (relatively free) profinite groups with tree-like Cayley graphs are constructed. Thus, new topologies on $F$ are found for which the Ribes-Zalesski{\u\i}-Theorem holds.

\end{abstract}

\maketitle
\section{Introduction} The Ribes--Zalesski{\u\i}-Theorem \cite{RZ} states that the product $H_1\cdots H_n$ of any finitely generated subgroups $H_1,\dots, H_n$ of a free group $F$ is closed in the profinite topology of $F$. The original motivation for this theorem came from a paper by Pin and Reutenauer \cite{PR}: the theorem provided a nice algorithm to compute the closure (with respect to the profinite topology) of a rational subset of a free group and  implied the truth of the {Rhodes type II conjecture}, then a long standing conjecture in the theory of finite monoids. Since then, the product theorem has become a subject of independent interest. It has been generalized in various directions: on the one hand to other groups \cite{coulbois,you}, on the other hand to other topologies of $F$. The original proof of the theorem by Ribes and Zalesski{\u\i} is formulated for the (full) profinite topology of $F$ but is valid  for the pro-$\mathfrak{C}$-topology of $F$ for  any extension closed variety $\mathfrak{C}$ of finite groups (in the sense that the product $H_1\cdots H_n$ is closed with respect to the pro-$\mathfrak{C}$-topology of $F$ for pro-$\mathfrak{C}$-closed $H_1,\dots, H_n$). This was generalized by Steinberg and the author \cite{geometry} to so-called \emph{arboreous} varieties, a class of varieties which is much larger than the class of all extension closed varieties. In that paper, tight connections between the pro-$\mathfrak{C}$-topology of a free group and the geometry of the Cayley graphs of the free pro-$\mathfrak{C}$-groups were established. The purpose of the present paper is twofold. On the one hand, it is aimed at extending the mentioned results of \cite{geometry} from varieties to formations. This includes to expand the method found in \cite{geometry} to construct profinite groups with tree-like Cayley graphs. On the other hand, a new direct approach, based on covering graphs, to the entire topic is presented.
The proofs in \cite{geometry} use an interpretation of the subject within varieties of finite inverse monoids and their relatively free profinite objects. Meanwhile it is the author's opinion that this approach obstructs the direct view and mystifies the immediate connection between the pro-$\fF$-topology of $F$ and the geometry of $\Gamma(\wh{F_\fF})$. In contrast to \cite{geometry}, the present proofs are direct and purely  geometric, without the use of (relatively free, profinite) inverse monoids; this should raise the accessibility of the results to readers not familiar with inverse monoids.

The paper is organized as follows. Section \ref{graphs} collects all results about graphs, finitely generated subgroups of free groups and profinite graphs that will be used  in the paper. Section \ref{geomvstop} studies two geometric properties the Cayley graph of a free pro-$\fF$-group for a formation $\fF$ might have: that of being Hall and that of being tree-like. It is then shown that these geometric properties are equivalent to separability properties induced on a free group by its pro-$\fF$-topology; in section \ref{ribeszalesskii} we establish the Ribes-Zalesski{\u\i}-Theorem for the pro-$\fF$-topology of a free group where $\fF$ is any formation whose free profinite objects have tree-like Cayley graphs. Finally, in section \ref{S-extensions} we define, for a finite, $A$-generated group $G$ and a finite simple group $S$, the \emph{$A$-universal $S$-extension of $G$} and show how this concept can be used to construct profinite groups with tree-like Cayley graphs and to guarantee that for certain formations $\fF$ the free pro-$\fF$-groups have tree-like Cayley graphs; such formations will be called \emph{arboreous}.

\section{Graphs, subgroups of free groups, and profinite graphs}\label{graphs}
\label{sec1}
We follow the Serre convention~\cite{Serre} and
define a \emph{graph} $\Gamma$ to consist of a set $V(\Gamma)$ of
\emph{vertices} and disjoint sets $E(\Gamma)$ of \emph{positively
oriented (or positive) edges} and $E\inv(\Gamma)$ of
\emph{negatively oriented (or negative) edges} together with
\emph{incidence} functions $\iota,\tau:E(\Gamma)\cup
E(\Gamma)\inv\rightarrow V(\Gamma)$ selecting the \emph{initial},
respectively, \emph{terminal} vertex of an edge $e$ and mutually
inverse bijections (both written: $e\mapsto e\inv$) between
$E(\Gamma)$ and $E\inv(\Gamma)$ such that $\iota e\inv = \tau e$
for all edges $e$ (whence $\tau e\inv =\iota e$, as well).  We set
$\til {E(\Gamma)}=E(\Gamma)\cup E\inv (\Gamma)$ and call it the
\emph{edge set} of $\Gamma$. Given this definition of a graph the
notions of \emph{ subgraph (spanned by a set of edges)}, \emph{
morphism of graphs} and \emph{ projective limit of graphs} have the
obvious meanings. Edges are to be thought of geometrically: when
one draws an oriented graph, one draws only the edge $e$  and
thinks of $e^{-1}$ as being the same edge, but traversed in the
reverse direction.

A \emph{path} $p$ in a graph $\Gamma$ is a finite sequence $p=e_1\dots
e_n$ of consecutive edges, that is $\tau e_i = \iota e_{i+1}$ for
all $i$; we define $\iota p= \iota e_1$ to be the initial vertex
of $p$ and $\tau p=\tau e_n$ to be the terminal vertex of $p$.  A path is \emph{reduced} if it does not contain a segment of
the form $ee\inv$ for any edge $e$. We also consider an empty path at each vertex.   A path $p=e_1\dots
e_n$ is a \emph{circuit at (base point)} $v$ if $v=\iota p=\tau p$.
 A graph is \emph{connected} if any two vertices can be joined by a
path. A connected graph is a \emph{tree} if it does not contain a non-empty reduced circuit.

Throughout, $A$ shall denote a finite alphabet, that is, a finite set of symbols called \emph{letters}, usually $\vert A\vert \ge 2$;
we use $A\inv$ to denote a disjoint copy of
$A$ consisting of formal inverses $a\inv$ of the letters $a$ of $A$, and set  $\til A=A\cup A\inv$.
An \emph{$A$-labelled graph} is a graph together with a labelling function $\ell:\til E\to \til A$ such that $\ell(e)\in A$  and $\ell(e\inv)=\ell(e)\inv$ for each positive edge $e$. A morphism of labelled graphs is assumed to respect the labelling. Given a path $p=e_1\dots e_n$ in a labelled graph, the label $\ell(p)$ of that path is just $\ell(e_1)\cdots\ell(e_n)$.

If $V(\Gamma)$, $E(\Gamma)$ and $E(\Ga)\inv$  are topological spaces,
$\til{E(\Gamma)}$ is the topological sum of $E(\Gamma)$ and
$E\inv(\Gamma)$, and $\iota$, $\tau$ and $(\ )\inv$ (in both directions) are continuous,
then $\Gamma$ is called a \emph{topological graph}. A
\emph{profinite graph} is a topological graph $\Gamma$ which is a
projective limit of finite, discrete graphs. It is well known
\cite{Mel, RZ3} that $\Gamma$ is profinite if and only if
$V(\Gamma)$ and $\til {E(\Gamma)}$ are both compact, totally
disconnected Hausdorff spaces.
Morphisms between profinite graphs are always assumed to be
continuous. \emph{Subgraphs} of profinite graphs are understood in the category of profinite graphs: they must be closed as topological spaces. Moreover, a \emph{connected profinite graph} is one all of whose finite continuous quotients are connected as abstract graphs (such are termed ``profinitely connected'' in \cite{AWred, AWsurvey}). For more informations about profinite graphs the reader is referred to \cite{AWred, AWsurvey, RZ3} and \cite[Section 2]{geometry}.

The profinite graphs of primary interest are subgraphs of Cayley graphs of
 profinite groups where a \emph{profinite group} is
a compact, totally disconnected group, or, equivalently, a
projective limit of finite groups. We refer the reader
to~\cite{RZbook1} for basic definitions on
profinite and relatively free profinite groups.

 Let $A$ be an alphabet, $G$ be a group and $\vp:A\to G$ be a map.
Then the \emph{Cayley graph of $G$ with respect to $(A,\vp)$}, or, if the mapping $\vp$ is clear, \emph{the Cayley graph of $G$ with respect to $A$},
denoted by $\Gamma _A(G)$, has vertex set $G$, edge set $G\times
\til A$, incidence functions given by $\iota (g,a) = g$, $\tau
(g,a) = g(a\vp)$ and involution $(g,a)\inv = (g(a\vp),a\inv)$. We call $a\in \til A$ the \emph{label} of $(g,a)$. The edge $(g,a)$ is usually drawn and thought of as $\underset{g}{\bullet}\!\overset{a}{\rbl}\!\!\!\!\!\underset{g(a\vp)}{\bullet}$. Throughout this paper we consider $A$-generated groups $G$ for a fixed alphabet $A$ and the mapping $\vp:A\to G$ is never mentioned; this means that $G$ is generated by $A\vp$ and $A$ is then treated like a subset of $G$ though distinct elements of $\til A$ are \emph{a priori} not necessarily distinct elements of $G$. In this case, $\Ga_A(G)$ is always denoted $\Ga(G)$ and is connected. A morphism from the $A$-generated group $G$ to the $A$-generated group $H$ is always a morphism extending the mapping $a\mapsto a$ and hence is uniquely determined and surjective and is called the \emph{canonical morphism} $G\twoheadrightarrow H$. Given a canonical morphism of finite $A$-generated groups $\vp:G\twoheadrightarrow H$ we have a canonical morphism $\Ga(G)\twoheadrightarrow \Ga(H)$, usually also denote $\varphi$, of finite $A$-labelled graphs which maps each vertex $g$ to $g\vp$ and each edge $(g,a)$ to $(g\vp,a)$. Suppose that $\{(G_i,\vp_{ij})\mid i,j\in I\}$ is an inverse system of finite $A$-generated groups; then automatically $\{(\Ga(G_i),\vp_{ij})\mid i,j\in I\}$ is an inverse system of finite $A$-labelled graphs. Let $\mathcal{G}=\vpl_{i\in I}G_i$; then the Cayley graph of $\mathcal G$ is $\vpl_{i\in I}\Ga(G_i)$ which is a profinite graph having vertex set $\mathcal G$ and edge set  $\mathcal{G}\times \til A$ where the topology of the latter is just the product topology with $\til A$ considered to be discrete. Since $A$ is a fixed finite set we need to consider only countable inverse systems $\{(G_n,\vp_{nm})\mid n,m\in \mathbb{N},m\le n\}$ where the linking morphism $\vp_{nm}$ are not always explicitly mentioned. We note that the Cayley graph of an \emph{$A$-generated profinite group} $\mathcal G$ (which means that the abstract subgroup $\left<A\right>$ of $\mathcal G$ generated by $A$ is dense in $\mathcal G$) is connected as a profinite graph.

Throughout, $F$ stands for the free group with basis $A$, its elements are represented as words in the alphabet $\til A$.  It is well known that finitely generated subgroups $H$ of $F$ can be encoded in terms of finite, labelled, pointed graphs \cite{kapmas,MSW,Stallings}. Let $\cA$ be a finite $A$-labelled graph; $\cA$ is \emph{folded} if for every letter $a\in A$ and every vertex $v$ there exists at most one edge $e$ starting or ending at $v$ and having label $a$. In a folded graph, the letters from $\til A$ induce partial injective mappings on the vertex set, and for every vertex $v$ and every word $w$ in the letters of $\til A$ there there is at most one path starting at $v$ and having label $w$. From now on we assume all graphs to be folded. Suppose that $\cA$ has a distinguished vertex $b$ (the base point) and let $L(\cA,b)$ be the set of all elements $w\in F$ ($w$ given as a reduced word in $\til A$) such that $w$ labels a closed path at $b$ in $\cA$. Then $L(\cA,b)$ is a finitely generated subgroup of $F$.

A  graph $\cA$ with base point $b$ is \emph{reduced} if no vertex except perhaps $b$ has degree $1$ (where the degree of a vertex $v$ is the number of positive edges $e$ for which $\iota e =v$ or $\tau e =v$). For every finitely generated subgroup $H$ of $F$ then there is up to isomorphism exactly one finite, connected, $A$-labelled, reduced graph $\cA$ with base point $b$ such that $L(\cA,b)=H$. This graph  we shall usually denote $\cH$ with base point $b_\cH$. It can be obtained as follows.  The \emph{Schreier graph} $\Sigma(F,H,A)$   has vertex set the set $H{\setminus}F$ of all right cosets $Hg$ in $F$ with respect to $H$ and positive edges $\underset{Hg}{\bullet}\!\overset{a}{\rb}\!\underset{Hga}{\bullet}$ for $g\in F$ and $a\in A$; the edge set then can be identified with $H{\setminus}F\times \til A$. Let the \emph{core graph} $\mathrm{core}(\Sigma(F,H,A),H)$ of $\Sigma(F,H,A)$ with respect to the base point $H$ be the subgraph of $\Sigma(F,H,A)$ spanned by edges which are contained in a reduced closed path at vertex $H$. Then $\mathrm{core}(\Sigma(F,H,A),H)$ is isomorphic to $\cH$ and will be the called the \emph{core graph} of $H$. For a more constructive method to find $\cH$ the reader is referred to \cite{kapmas,MSW, Stallings}. In any case, the core graph of $H=L(\cA,b)$, $\cA$ finite, can be obtained from $\cA$ by removing finitely may vertices and edges from $\cA$ until the outcome is reduced.

Next let $\mathfrak{F}$ be a formation of finite groups; the pro-$\mathfrak{F}$-topology of $F$ has as neighborhood basis of $1$ the collection of all normal subgroups $N$ of $F$ for which $F/N\in \mathfrak{F}$.  The profinite group $\wh{F_{\fF}}:=\vpl_{F/N\in \mathfrak{F}}F/N$ is the free pro-$\fF$-group generated by $A$; the abstract subgroup of $\wh{F_{\fF}}$ generated by $A$ coincides with $F$ if and only if $F$ is residually $\fF$, that is, if and only if $\bigcap_{F/N\in \fF}N=\{1\}$. Given an $A$-generated group $G\in \fF$ and $w\in F$ it is often convenient to denote the image of $w$ under the canonical morphism $F\twoheadrightarrow G$ by $[w]_G$ (which means the \emph{value} of $w$ in $G$); similarly, given $\ga\in \wh{F_{\fF}}$, the image of $\ga$ under the canonical morphism $\wh{F_{\fF}}\twoheadrightarrow G$ will be denoted $[\ga]_G$.

A finite $A$-labelled graph $\cA$ is \emph{complete} if the degree of every vertex is $2\vert A\vert$. That is, every letter $a\in \til A$ induces a permutation of the set of vertices of $\cA$. The group generated by these permutations is the \emph{transition group} $T_{\cA}$ of $\cA$ which is an $A$-generated group. A finitely generated subgroup $H$ of $F$ is \emph{$\mathfrak{F}$-extendible} \cite{MSW} if the core graph $\cH$ of $H$ can be embedded into a finite complete graph $\ol{\cH}$ whose transition group $T_{\ol{\cH}}$ belongs to $\mathfrak{F}$. A key result from \cite{MSW} is Propositon 2.7: it  states that every pro-$\mathfrak{F}$-closed finitely generated subgroup $H$ of $F$ is $\mathfrak{F}$-extendible --- the result is formulated and proved for varieties of finite groups but the proof goes through verbally for the more general case of formations.

Suppose that the complete graph $\cA$ with distinguished vertex $b$ is connected; then there exists a unique graph morphism $\vp_{\cA}:\Ga(T_{\cA})\twoheadrightarrow \cA$ for which $1\vp_{\cA}=b$ which we call the \emph{canonical morphism} $\Ga(T_{\cA})\twoheadrightarrow \cA$. Indeed, it is well known and easy to see that the mapping $g\mapsto b\cdot g$ induces a graph morphism $\Ga(T_{\ol \cA})\twoheadrightarrow \cA$ and $\cA$ is isomorphic to the Schreier graph $\Sigma(T_\cA,H,A)$ where $H$ is the stabilizer of $b$ in $T_{\cA}$.
If $G$ is another $A$-generated group such that $\vp:G\twoheadrightarrow T_{\cA}$ then $\vp\vp_{\cA}:\Ga(G)\twoheadrightarrow \cA$ is the \emph{canonical morphism $\Ga(G)\twoheadrightarrow \cA$}.
Next let $\cA$ be an $A$-labelled graph with base point $b$ and $\ol \cA$ be a completion of $\cA$ (that is, $\cA$ is a subgraph of the complete graph $\ol\cA$). We have the canonical mapping $\vp_{\ol{\cA}}:\Ga(T_{\ol{\cA}})\twoheadrightarrow\ol{\cA}$. Now, the inverse image of $\cA$ under $\vp_{\ol{\cA}}$ in $\Ga(T_{\ol{\cA}})$ is a subgraph of $\Ga(T_{\ol{\cA}})$, not necessarily connected. But we let ${\cA}^{T_{\ol{\cA}}}$ be the connected component containing $1$ of this inverse image. More generally, let $G$ be $A$-generated with canonical morphism $\vp:G\twoheadrightarrow T_{\ol{\cA}}$; then we define $\cA^G$ to be the connected component containing $1$ of the inverse image of $\cA$ under the canonical map $\vp\vp_{\ol{\cA}}:\Ga(G)\twoheadrightarrow \ol{\cA}$. Then $\cA^G$ is a connected subgraph of $\Ga(G)$ and it is the subgraph of $\Ga(G)$ spanned by all edges that are contained in a path starting at $1$ and having label $w$, say,  for which there exists a path in $\cA$ starting at $b$ and having label $w$. The latter could be also taken as definition of the graph $\cA^G$; it actually makes sense for every $A$-generated group $G$, but the existence of a morphism $\cA^G\twoheadrightarrow \cA$ is not guaranteed in the general case. However, if there is a completion $\ol\cA$ of $\cA$ and $G\twoheadrightarrow T_{\ol\cA}$ then we do have a canonical graph morphism $\cA^G\twoheadrightarrow \cA$ mapping $1$ to $b$. For every $h\in L(\cA,b)$ the element $[h]_G$ is also mapped to $b$ by this morphism. Moreover, as a subgraph of $\Ga(G)$ the graph $\cA^G$ may be shifted by left multiplication by an element $g\in G$ to obtain $g\cA^G$. Then there is a canonical graph morphism $g\cA^G\twoheadrightarrow \cA$ which maps $g$ to $b$. Similarly, for every $h\in L(\cA,b)$, the element $g[h]_G$ is mapped to $b$ under this morphism.

If we are given an inverse system of $A$-generated groups
\begin{equation}\label{invGrou}
T_{\ol\cA}\overset{\vp_0}{\twoheadleftarrow} G_0\overset{\vp_1}{\twoheadleftarrow} G_1\overset{\vp_2}{\twoheadleftarrow} G_2\overset{\vp_3}{\twoheadleftarrow} \cdots
\end{equation}
then we get an inverse system of $A$-labelled graphs
\begin{equation}\label{invGra}
\cA\overset{\vp_{\ol\cA}}{\twoheadleftarrow}{\cA} ^{{T_{\ol\cA}}}\overset{\vp_0}{\twoheadleftarrow}{\cA}^{G_0}\overset{\vp_1}{\twoheadleftarrow}\cA^{G_1}
\overset{\vp_2}{\twoheadleftarrow}\cA^{G_2}\overset{\vp_3}{\twoheadleftarrow}\cdots\end{equation}
where in the latter sequence the mappings $\vp_{\ol\cA}$ and $\vp_i$ have to be understood as the appropriate restrictions to the graphs ${\cA}^{T_{\ol\cA}}$ and $\cA^{G_i}$. Altogether,  every inverse system of finite groups as in (\ref{invGrou}) leads to an inverse system  of finite graphs as in (\ref{invGra}) and for $\mathcal {G}=\vpl G_n$ we set $\cA^{\mathcal{G}}:=\vpl \cA^{G_n}$ which is a connected subgraph of the profinite Cayley graph $\Ga(\mathcal{G})$.
Assume that the abstract subgroup of $\mathcal G$ generated by $A$ is $F$.
Similarly as for the finite case, we have now a canonical morphism $\cA^{\mathcal{G}}\twoheadrightarrow \cA$ which maps every element of the closure $\ol{L(\cA,b)}$ of $L(\cA,b)$ in $\mathcal{G}$ to the base point $b$. In addition, for every $\ga\in \mathcal{G}$ there is a canonical graph morphism $\ga\cA^{\mathcal{G}}\twoheadrightarrow \cA$ which maps $\ga\ol{L(\cA,b)}$ to $b$.

As mentioned earlier, the graphs $\cA^{G_i}$  may be defined without referring to a completion $\ol\cA$ of $\cA$. The same holds for the projective limit $\cA^{\mathcal G}=\varprojlim \cA^{G_i}$. The profinite graph $\cA^{\wh{F_{\fF}}}$ defined in this fashion then admits a canonical graph morphism $\cA^{\wh{F_{\fF}}}\twoheadrightarrow\cA$ if and only if $\cA$ admits a completion $\ol\cA$ whose transition group $T_{\ol{\cA}}$ is in $\fF$. The `only if' direction of that claim --- which will not be used in the paper --- can be seen as follows. Take an inverse system $(G_n)_{n\in \mathbb{N}}$ with $G_n\in \mathfrak{F}$ such that $\wh{F_{\fF}}=\vpl G_n$. Then $\cA^{\wh{F_{\fF}}}=\vpl \cA^{G_n}$; hence the canonical morphism $\cA^{\wh{F_{\fF}}}\twoheadrightarrow \cA$ factors through $\cA^{G}$ for $G=G_n$ for some $n$.
Now let $H$ be the subgroup of $G$ consisting of all $[w]_G$ for $w\in L(\cA,b_\cA)$. The Schreier graph $\Sigma(G,H,A)=:\Sigma$ then is a complete graph which contains $\cA$ as subgraph. The transition group $T_\Sigma$ is a quotient of $G$, namely $T_\Sigma\cong G/{H_G}$ where $H_G$ is the core of $H$ in $G$.
The profinite graph $\cA^{\wh{F_{\fF}}}$ can be viewed as \emph{pro-$\fF$-universal covering graph of $\cA$}.

\section{Profinite topology and geometry of graphs}
For a certain class of formations $\mathfrak{F}$ of finite groups we establish tight connections between the pro-$\mathfrak F$-topology of a finitely generated free group $F$ and the geometry of the Cayley graph $\Gamma(\wh{F_\mathfrak{F}})$ of its pro-$\mathfrak F$-completion $\wh{F_\mathfrak{F}}$. This is an extension to formations of the results formulated and proved for varieties in \cite[Section 5]{geometry}. While the proofs of \cite{geometry} would carry over to the case of formations more or less easily, it is our explicit intention to present new proofs which avoid relatively free profinite inverse monoids and instead use the above mentioned concept of (profinite) covering graphs.
\subsection{Being Hall and being arboreous}\label{geomvstop}
According to a result of \cite{MSW} mentioned in section \ref{sec1}, for every formation $\mathfrak{F}$, every finitely generated pro-$\mathfrak{F}$-closed subgroup $H$ of $F$ is $\mathfrak{F}$-extendible. Here we shall present a criterion for an extendible group $H$ to be closed. It turns out that this can be expressed in terms of a geometric property of the pro-$\mathfrak{F}$-universal covering graph $\cH^{\wh {F_\fF}}$ of the core graph $\cH$ of $H$.

For this purpose, we modify and generalize the concept of \emph{Hall property} of a profinite graph introduced in \cite{geometry}. Let $\Ga$ be a connected profinite graph; a connected subgraph $\De$ is a \emph{Hall subgraph} of $\Ga$ if, whenever $\De$ contains the endpoints of a finite reduced path $p$ in $\Ga$ then $\De$ contains the (graph spanned by the) path $p$ itself. Note that, by definition, every connected profinite graph is a Hall subgraph of itself.  The concept of Hall subgraph allows to give a characterization of which  $\fF$-extendible subgroups $H$ of $F$ are pro-$\fF$-closed.
\begin{Thm}\label{Hall:individual} Let $\fF$ be a formation of finite groups. An $\fF$-extendible subgroup $H$ of $F$ is pro-$\fF$-closed if and only if $\cH^{\wh{F_\fF}}$ is a Hall subgraph of $\Ga(\wh{F_\fF}).$
\end{Thm}
\begin{proof} Here we do not assume that $F$ is residually $\fF$. In particular we do not assume that  $F$ is canonically embedded in $\wh{F_\fF}$. For convenience, for a word $w\in F$, the value of $w$ in $\wh{F_\fF}$ will simply be denoted $[w]_\fF$ rather than $[w]_{\wh{F_\fF}}$.

Let $H$ be an  $\fF$-extendible subgroup of $F$ with core graph $\cH$. Suppose first that $\cH^{\wh{F_\fF}}$ is not a Hall subgraph of $\Ga(\wh{F_\fF})$. We need to show that $H$ is not pro-$\fF$-closed. Throughout the following proof we shall identify every path $p$ with the graph spanned by $p$. Let $\ga$ be a vertex of $\cH^{\wh{F_\fF}}$ for which there exists a reduced word $w\in F$ such that the path  $\ga\to \ga[w]_\fF$ in $\Ga(\wh{F_\fF})$, starting at $\ga$ and being labelled $w$, is outside $\cH^{\wh{F_\fF}}$ except for its endpoints.

Let $\ol{\cH}$ be a completion of $\cH$ with transition group $T$ in $\fF$ and let $G_0\twoheadleftarrow G_1\twoheadleftarrow\cdots$ be an inverse system of groups in $\fF$ such that $\varprojlim G_n=\wh{F_\fF}$ and $G_0\twoheadrightarrow T$, and such that, denoting for all $n$ the canonical projections $\wh{F_\fF}\twoheadrightarrow G_n$ by $\varphi_n$, we have that the path $\ga\varphi_0\to \ga\varphi_0[w]_{G_0}$ in $\Ga(G_0)$ labelled $w$ is outside $\cH^{G_0}$ (except for its endpoints). Let $v\in F$ be a reduced word labelling a path $1\to \ga\varphi_0$ running entirely inside $\cH^{G_0}$. Then, in $\Ga(\wh{F_\fF})$, the path $1\to [v]_\fF$ labelled $v$ runs inside $\cH^{\wh{F_\fF}}$ (being the lift of the corresponding path in $\cH^{G_0}$); in particular, $[v]_\fF\in \cH^{\wh{F_\fF}}$. Next, the path $[v]_\fF\to [vw]_\fF$ in $\Ga(\wh{F_\fF})$ labelled $w$ is outside $\cH^{\wh{F_\fF}}$ (except, perhaps, for its endpoints) since its projection in $\Ga(G_0)$ is outside $\cH^{G_0}$. We claim that $[vw]_\fF$ does belong to $\cH^{\wh{F_\fF}}$. Indeed, for each $n$ choose a word $p_n$ labelling a path $\ga\varphi_n\to (\ga[w]_\fF)\varphi_n=(\ga\varphi_n)[w]_{G_n}$ which runs entirely inside $\cH^{G_n}$. Since $[p_n]_{G_n}=[w]_{G_n}$ for all $n$, $\lim p_n=w$.  Each path $1\to[vp_n]_\fF$ in $\Ga(\wh{F_\fF})$ labelled $vp_n$ runs entirely in $\cH^{\wh{F_\fF}}$ being the lift of the corresponding path in $\cH^{G_n}$. Since $\cH^{\wh{F_\fF}}$ is closed it follows that $[vw]_\fF=\lim[vp_n]_\fF\in \cH^{\wh{F_\fF}}$. Hence we may consider the vertices $[v]_\fF$ and $[vw]_\fF$ instead of $\ga$ and $\ga[w]_\fF$.

Next, consider the group $K:=v\inv Hv$. Since conjugation is a homeomorphism of $F$, $H$ is closed if and only if $K$ is closed. It is clear that for the core graph $\cK$ of $K$ the equality $\cK^{\wh{F_\fF}}=v\inv \cH^{\wh{F_\fF}}$ holds. For $\cK^{\wh{F_\fF}}$ we have that $1,[w]_\fF\in \cK^{\wh{F_\fF}}$ but the path $1\to[w]_\fF$ in $\Ga(\wh{F_\fF})$ labelled $w$ is outside $\cK^{\wh{F_\fF}}$ (except for its endpoints). Finally, choose, for every $n$, a word $u_n$ which labels a path $1\to [w]_{G_n}$ which runs entirely inside $\cK^{G_n}$. Then $\lim u_n=w$ and $u_nu\inv_0\in K$ since this word labels a closed path at $1$ in $\cK^{G_0}$. On the other hand, $\lim u_nu_0\inv=wu_0\inv\notin K$: if $wu_0\inv$ were in $K$ then it would label a closed path in $\cK$ at base point $b_\cK$ and that path would lift to a path in $\cK^{G_0}$ starting at $1$. This however, this is not the case by construction of $w$.

Let conversely $H\le F$ be finitely generated, $\fF$-extendible with core graph $\cH$ and suppose that $\cH^{\wh{F_\fF}}$ is a Hall subgraph of $\Ga(\wh{F_\fF})$. Let $(w_n)_{n\in \mathbb{N}}$ be a sequence of elements of $H$ with $\lim w_n=w\in F$ with respect to the pro-$\mathfrak{F}$-topology of $F$. We need to show that $w\in H$.  Since $[w_n]_\fF\in \cH^{\wh{F_{\fF}}}$ for all $n$, also $[w]_\fF\in\cH^{\wh{F_{\fF}}}$. So, $\cH^{\wh{F_{\fF}}}$ contains $1$ and $[w]_\fF$ and therefore also the path $1\to [w]_\fF$ in $\Ga(\wh{F_\fF})$ labelled $w$. By the canonical map $\cH^{\wh{F_{\fF}}}\twoheadrightarrow\cH$, that path is mapped to a closed path at $b_\cH$ with label $w$ which means that $w\in H$.
\end{proof}

Of particular interest are formations $\fF$ for which \textbf{every} extendible subgroup $H$ of $F$ is closed. Let us call such a formation \emph{Hall}. More precisely,
a formation $\mathfrak{F}$ is \emph{Hall} if for every alphabet $A$ (with $\vert A\vert\ge 2$) every finitely generated, $\fF$-extendible subgroup $H$ of $F$ is pro-$\fF$-closed in $F$.
This is the straightforward generalization to formations of a notion originally introduced for varieties \cite{myRZ}. Exactly as in that special case, the property of being Hall of a formation $\mathfrak F$ can be expressed in terms of the geometry of the Cayley graphs of free pro-$\fF$-groups. A connected profinite graph $\Ga$ is an \emph{absolute Hall graph}, or, shortly a \emph{Hall graph} if every connected subgraph is a Hall subgraph of $\Ga$. If $\Ga(\wh{F_\fF})$ is a Hall graph then by Theorem \ref{Hall:individual}, every $\fF$-extendible subgroup $H$ of $F$ is pro-$\fF$-closed. Also the converse holds: if $\Ga(\wh{F_\fF})$ is not Hall then there exists an $\fF$-extendible subgroup $H$ of $F$ which is not pro-$\fF$-closed. This leads to a `global' version of Theorem \ref{Hall:individual}; for varieties this was proven in \cite{geometry}

\begin{Thm}\label{Hallformations} A formation $\mathfrak{F}$ is Hall if and only if for each alphabet $A$ the Cayley graph $\Ga(\wh{F_{\fF}})$ is Hall.
\end{Thm}
\begin{proof} The `if' part is a consequence of Theorem \ref{Hall:individual}. For the converse, let $\De\subseteq\Ga(\wh{F_{\fF}})$ be a connected subgraph containing the endpoints $\ga$ and $\ga[w]_\fF$ of some  reduced path $p$ whose label is $w$, say, such that the graph spanned by $p$ is outside $\De$ except for its endpoints.  We may consider instead the graph $\ga\inv\De$ and hence assume that $\ga=1$. As in the proof of Theorem \ref{Hall:individual}, let $G_0\twoheadleftarrow G_1\twoheadleftarrow \cdots$ be an inverse system of groups in $\fF$ such that $\varprojlim G_n=\wh{F_\fF}$ and, denoting the canonical morphism $\wh{F_\fF}\twoheadrightarrow G_n$ by $\varphi_n$ we have that the graph spanned by the path in $\Ga(G_0)$ with label $w$  starting at $1$ is outside $\De\varphi_0$ except for the endpoints $1$ and $[w]_{G_0}$. For every $n$, let $u_n$ be a word labelling a path $1\to [w]_{G_n}$ running inside $\De\varphi_n$; then $\lim u_n=w$. The group $H=L(\De\varphi_0,1)$ is $\fF$-extendible; since $u_n$ also labels a path $1\to [w]_{G_0}$ in $\De\varphi_0$ it follows that $u_nu_0\inv\in H$ for every $n$ and $\lim u_nu_0\inv=wu_0$. But $wu_0\inv$ does not belong to $H$ since $wu_0\inv$ does not label a closed path at $1$ in $\De\varphi_0$.
Altogether, $H$ is not closed for the pro-$\fF$-topology.
\end{proof}

Next we consider a strengthening of the former property, this time motivated from the side of the geometry of the Cayley graph of $\wh{F_\fF}$. An obvious strengthening of the property of being Hall is that any two vertices $u$ and $v$ which are contained in a connected subgraph should be contained in a smallest (with respect to containment) connected subgraph. Equivalently, for  any two given vertices contained in a connected subgraph, the intersection of all such connected subgraphs is again connected. A connected profinite graph with this property has been called \emph{tree-like} in \cite{geometry}; pro-$p$-trees  and similarly defined graphs in the sense of \cite{RZ,RZ2,RZ3} do have this property but there are many tree-like (Cayley) graphs which are not of this form (including some new examples which can be constructed by the method presented in section \ref{S-extensions}). We are going to show that this geometric property of the Cayley graph is equivalent to a strengthened condition on the pro-$\fF$-topology of the free group $F$: namely that the product $HK$ of any two $\fF$-extendible subgroups $H$ and $K$ of $F$ is pro-$\fF$-closed. Throughout, for vertices $\alpha, \beta$ of a tree-like graph $\Ga$ we shall denote by $[\alpha,\beta]$ the unique smallest connected subgraph of $\Ga$ containing $\alpha$ and $\beta$  ---  the \emph{geodesic} subgraph determined by $\alpha$ and $\beta$.

\begin{Prop}\label{tlidc} Suppose that $\Ga(\wh{F_{\fF}})$ is tree-like; then the product $HK$ of any two $\fF$-extendible subgroups $H$ and $K$ of $F$ is closed in the pro-$\fF$-topology of $F$.
\end{Prop}

\begin{proof} Let $(w_n)$ be a sequence in $HK$ such that $\lim w_n=w\in F$; we need to prove that $w\in HK$. Each $w_n$ admits a representation $w_n=h_nk_n$ with $h_n\in H$ and $k_n\in K$  (equality holds in $F$, so the word $h_nk_n$ need not be reduced). We may assume (by going to a subsequence) that the sequences $(h_n)$ and $(k_n)$ converge in $\wh{F_{{\fF}}}$: there exist $\eta,\ka\in \wh{F_{{\fF}}}$ such that $\eta=\lim h_n$ and $\ka=\lim k_n$. Then $\eta\kappa=w$ and we may assume that $\eta\ne 1\ne w$. Let $\cH$ and $\cK$, respectively, be the core graphs of $H$ and $K$. By construction, $\cH^{\wh{F_{\fF}}}$ contains $1$ and all vertices $h_n$ and hence also $\eta=\lim h_n$. Likewise, $\cK^{\wh{F_{\fF}}}$ contains $1$ and $\kappa$. Consequently, $\cH^{\wh{F_{\fF}}}\cup\eta\cK^{\wh{F_{\fF}}}$ contains $1$ and $w=\eta\kappa$. By the Hall property, $\cH^{\wh{F_{\fF}}}\cup\eta\cK^{\wh{F_{\fF}}}$ contains the geodesic $[1,w]$. Since $1\in \cH^{\wh{F_{\fF}}}$ and $w\in \eta\cK^{\wh{F_{\fF}}}$, some vertex, say $v$, of the geodesic $[1,w]$ is contained in $\cH^{\wh{F_{\fF}}}\cap\eta\cK^{\wh{F_{\fF}}}$. That is, $w$ admits a factorization $w=vu$ such that $v\in \cH^{\wh{F_{\fF}}}\cap\eta\cK^{\wh{F_{\fF}}}$. Since $\Ga(\wh{F_\fF})$ is tree-like, we have
$$[1,v]\subseteq \cH^{\wh{F_\fF}},\quad [v,\eta]\subseteq \cH^{\wh{F_{\fF}}}\cap \eta\cK^{\wh{F_{\fF}}},\quad [v,w]\subseteq \eta\cK^{\wh{F_\fF}}.$$
\begin{figure}[ht]
\begin{tikzpicture}
\draw[thick] (0,1) to [out=90,in=180] (2,2) to [out=0,in=90] (4,1) to [out=-90,in=0]  (2,0) to [out=180,in=-90] (0,1);
\draw[thick] (2.5,1) to [out=90,in=180] (4.5,2) to [out=0,in=90] (6.5,1) to [out=-90,in=0]  (4.5,0) to [out=180,in=-90] (2.5,1);
\filldraw (1,1) circle (2pt); \filldraw (3.25,1) circle (2pt); \filldraw (3.25,0.6) circle (2pt);\filldraw (5.5,1) circle (2pt);
\node[below] at (1,1) {$1$}; \node[below] at (5.5,1) {$w$}; \node[above] at (3.25,1) {$\eta$}; \node[below] at (3.25,0.6) {$v$};
\draw (1,1) -- (3.25,0.6) -- (5.5,1);
\end{tikzpicture}
\caption{The graph $\cH^{\wh{F_{\fF}}}\cup\eta\cK^{\wh{F_{\fF}}}$.}
\end{figure}

Now consider a finite quotient $\wh{F_\fF}\overset{\vp}{\twoheadrightarrow} G$ such that:
\begin{itemize}
\item $G\twoheadrightarrow T_{\ol{\cH}}$ and $G\twoheadrightarrow T_{\ol{\cK}}$,
\item $1\ne \eta\vp\ne w\vp$,
\item $\vp$ restricted to $[1,w]$ is injective,
\end{itemize}
where $\ol\cH$ and $\ol\cK$ are completions of $\cH$ and $\cK$, respectively, with transition groups in $\fF$.  Then we have
$$[v,\eta]\vp\subseteq  (\cH^{\wh{F_{\fF}}})\vp=\cH^G
\mbox{ and }
[v,\eta]\vp\subseteq  (\eta\cK^{\wh{F_{\fF}}})\vp=(\eta\vp)\cK^G.$$
Choose a path $p:v\vp\to {\eta}\vp$ which runs entirely inside $[v,\eta]\vp$ and suppose its label is $s$. Then the word $vs$ labels a path $1\to{\eta}\vp$ running inside $\cH^G$. Since the projection $\cH^{\wh{F_{\fF}}}\twoheadrightarrow\cH$ maps both $1$ and $\eta$ to the base point $b_\cH$ of $\cH$  so does the projection $\cH^G\twoheadrightarrow \cH$:  $1$ and ${\eta}\vp$ are both mapped to the base point $b_\cH$. It follows that $vs\in H$. By an analogous reasoning: $s\inv u$ labels a path ${\eta}\vp\to w\vp$ running inside $(\eta\vp)\cK^G$. The canonical projection $(\eta\vp)\cK^G\twoheadrightarrow \cK$ maps both ${\eta}\vp$ and $w\vp$ to the base point $b_\cK$ of $\cK$. Altogether, $s\inv u\in K$ and thus $w=vs\cdot s\inv u\in HK$, as required.
\end{proof}

Next we are going to prove the converse of Proposition \ref{tlidc}. We start with a definition which will be also used in Section \ref{S-extensions}. Let $\Ga$ be a (pro)finite graph with distinguished vertex $1$; a \emph{constellation} in $\Ga$ is a triple $(\Xi,\ga,\Theta)$ where
\begin{enumerate}
\item $\ga$ is a vertex of $\Ga$, $\Xi$ and $\Theta$ are connected subgraphs of $\Ga$,
\item $1,\ga\in \Xi\cap \Theta$,
\item the connected $\Xi\cap \Th$-components of $1$ and $\ga$ are distinct.
\end{enumerate}
A constellation of an $A$-generated group then is a constellation of its Cayley graph. It is clear that the Cayley graph $\Ga(\mathcal G)$ of an $A$-generated pro-finite group $\mathcal G$ is tree-like if and only if it does not admit a constellation.

Let $G_0\twoheadleftarrow G_1\twoheadleftarrow G_2\twoheadleftarrow\cdots$ be an inverse system of $A$-generated finite groups with $\mathcal{G}=\vpl G_n$ with canonical projections $\mathcal{G}\overset{\vp_n}{\to}G_n$. Suppose that $\Ga(\mathcal G)$ admits a constellation $(\Xi,\ga,\Th)$. Then there exists a  positive integer $m$ such that $(\Xi\vp_m,\ga\vp_m,\Th\vp_m)$ is a constellation in $\Ga(G_m)$ (since the two $\Xi\cap\Th$-components of $1$ and $\ga$ may be separated by clopen subgraphs). Moreover, for every $n\ge m$, denoting by $\vp_{nm}$ the canonical morphism $G_n\to G_m$ we see that $(\Xi\vp_n,\ga\vp_n,\Th\vp_n)$ is a constellation in $\Ga(G_n)$ and
$$(\Xi\vp_m,\ga\vp_m,\Th\vp_m)=(\Xi\vp_n\vp_{nm},\ga\vp_n\vp_{nm},\Th\vp_n\vp_{nm}).$$

\begin{Prop}\label{rz2} If $\Ga(\wh{F_{\fF}})$ is not tree-like then there exist $\fF$-extendible subgroups $H$ and $K$ of $F$ for which the product $HK$ is not pro-$\fF$-closed in $F$.
\end{Prop}
\begin{proof} Take a constellation $(\Xi,\ga,\Th)$ in $\Ga(\wh{F_{\fF}})$ and a finite quotient $G$ of $\wh{F_{\fF}}$ with morphism $\vp:\wh{F_{\fF}}\twoheadrightarrow G$ such that $(\Xi\vp,\ga\vp,\Th\vp)$ is a constellation in $\Ga(G)$. Set $G_0:=G$ and take an inverse system $G_0\twoheadleftarrow G_1\twoheadleftarrow G_2\twoheadleftarrow \cdots$ of $A$-generated groups with $\wh{F_{\fF}}=\vpl G_n$. Denote by $\vp_m$ the canonical morphism $\wh{F_{\fF}}\twoheadrightarrow G_m$. For every $m$, the triple $(\Xi\vp_m,\ga\vp_m,\Th\vp_m)$ is a constellation in $\Ga(G_m)$.

Choose two words $u,v\in F$ such that $u$ labels a path $1\to \ga\vp_0$ which runs entirely in $\Xi\vp_0$ and $v$ labels a path $\ga\vp_0\to 1$ which runs entirely in $\Th\vp_0$. Next, for each $m$ choose words $h_m, k_m\in F$ such that $h_m$ labels a path $\ga\vp_m\to 1$ in $\Ga(G_m)$ which runs entirely in $\Xi\vp_m$ and $k_m$ labels a path $1\to\ga\vp_m$ in $\Ga(G_m)$ which runs entirely in $\Th\vp_m$, Then, for every $m$,
$$[h_m]_{G_m}=[\ga\inv]_{G_m}\mbox{ and }[k_m]_{G_m}=[\ga]_{G_m}$$
and therefore $\lim h_m=\ga\inv$ and $\lim k_m=\ga$ in $\wh{F_{\fF}}$.

Notice that $h_m$ labels also a path $\ga\vp_0\to 1$ running in $\Xi\vp_0$ and $k_m$ labels a path $1\to\ga\vp_0$ running in $\Th\vp_0$.
Let  $H=L(\Xi\vp_0,1)$ and $K=L(\Th\vp_0,1)$, respectively. Both $H$ and $K$ are $\fF$-extendible. Since $uh_m$ labels a closed path at $1$ in $\Xi\vp_0$ and $k_mv$ labels a closed path at $1$ in $\Th\vp_0$ it follows that $uh_m\in H$ and $k_mv\in K$. Moreover,
$$\lim_{m\to \infty}uh_mk_mv=u\ga\inv\ga v=uv\in F.$$
We are left with showing that $uv\notin HK$. Suppose, by contrast, that there are $h\in H, k\in K$ such that $uv=hk$ (equality holding in $F$). Then $h\inv u=kv\inv$. In $G_0$ we have
$$[h\inv u]_{G_0}=\ga\vp_0=[kv\inv]_{G_0},$$
 and $h\inv u$ and therefore also its reduced form $\mathrm{red}(h\inv u)$ labels a path $1\to \ga\vp_0$ in $\Ga(G_0)$ running entirely in $\Xi\vp_0$; likewise, $kv\inv$ and therefore also $\mathrm{red}(kv\inv)$ labels a path $1\to\ga\vp_0$ in $\Ga(G_0)$ running entirely in $\Th\vp_0$. However, $\mathrm{red}(h\inv u)$ and $\mathrm{red}(kv\inv)$ are identical as words and so must label the same path $1\to\ga\vp_0$ in $\Ga(G_0)$. This contradicts the fact that there is no path $1\to \ga\vp_0$ in $\Ga(G_0)$ running entirely in $\Xi\vp_0\cap \Th\vp_0$.
\end{proof}

As a consequence we may summarize:

\begin{Thm}\label{arboreous} The free pro-$\fF$-group $\wh{F_{\fF}}$ has tree-like Cayley graph if and only if for any two  $\fF$-extendible subgroups $H$ and $K$ of $F$ the product $HK$ is pro-$\fF$-closed.
\end{Thm}

The last result motivates the following definition: a formation $\fF$ is \emph{arboreous} if the Cayley graph $\Ga(\wh{F_{\fF}})$ of every finitely generated free pro-$\fF$-group is tree-like; this is the straightforward analogue in the context of formations of a notion introduced for varieties in \cite{geometry}. Then $\fF$ is arboreous if and only if for every finitely generated free group $F$, the product $HK$ of any two  $\fF$-extendible subgroups $H$ and $K$ of $F$ is closed in $F$ with respect to the pro-$\fF$-topology.

It would be interesting if in Proposition \ref{rz2} and therefore in the `if'-part of Theorem \ref{arboreous} ``$\fF$-extendible'' may be replaced with ``pro-$\fF$-closed'' for the involved groups $H$ and $K$. If we allow formations $\fF$ for which $F$ is not residually $\fF$ then this is not true for trivial reasons. Indeed, in that case the Cayley graph $\Ga(\wh{F_\fF})$ is certainly not tree-like since it contains a finite circuit. Let $L_\fF$ be the intersection of all normal subgroups $N$ of $F$ for which $F/N\in\fF$. Then the pro-$\fF$-closure of $\{1\}$ is $L_\fF$ and every pro-$\fF$-closed subgroup $H$ of $F$ must contain $L_\fF$. In that case, the core graph of $H$ is complete, hence $H$ is --- if it is finitely generated --- of finite index and therefore clopen. The product $HK$ of two such groups is necessarily also clopen (being the union of finitely many cosets of $H$). In this case the adequate environment to study separability properties concerning products $H_1\cdots H_n$ of finitely generated subgroups is the group $F/L_\fF$ rather than $F$.

\begin{Problem}  Does there exist a non-arboreous formation (or variety) $\fF$ for which $F$ is residually $\fF$ and such that the product $HK$ is pro-$\fF$-closed for any two finitely generated pro-$\fF$-closed subgroups $H$ and $K$ of $F$?
\end{Problem}
Any such example must be a non-Hall formation. Another open problem occurring in this context is:
\begin{Problem}  Does there exist a formation (or variety) which is Hall but not arboreous?
\end{Problem}

\subsection{The Ribes-Zalesski{\u\i}-Theorem}\label{ribeszalesskii} Now we show that for an arboreous formation $\fF$ even a stronger separability conditon holds: the product $H_1\cdots H_n$ of an arbitrary number $n$ of $\fF$-extendible subgroups $H_1,\dots, H_n$ of $F$ is pro-$\fF$-closed.

We start with some preliminary statements concerning tree-like graphs.

\begin{Lemma}\label{disconnected} Let $\Gamma$ be a connected profinite graph and $\Delta$ and $\Xi$ be non-empty disjoint subgraphs of $\Gamma$. Then the graph $\Delta\cup \Xi$ is disconnected.
\end{Lemma}
\begin{proof} The closed subsets $V(\De)$ and $V(\Xi)$ of $V(\Ga)$ are disjoint and therefore can be separated by open sets. Hence there exists a finite quotient $\Gamma'$ of $\Gamma$ with canonical morphism $\varphi:\Gamma\twoheadrightarrow \Gamma'$ such that $\Delta\varphi\cap\Xi\varphi=\emptyset$. Since $\Delta\varphi$ and $\Xi\varphi$ have no common vertex, every path from a vertex in $\Delta\varphi$ to a vertex in $\Xi\varphi$ must traverse an edge outside $\Delta\varphi\cup\Xi\varphi$. It follows that $(\Delta\cup\Xi)\varphi=\Delta\varphi\cup \Xi\varphi$ is not connected. Consequently, $\Delta\cup \Xi$ is not connected, as well.
\end{proof}
\begin{Thm}\label{circuitforgraphs} Let $\Gamma$ be a tree-like graph, $n\ge 3$ and $\gamma_0,\dots,\gamma_n\in V(\Gamma)$ be such that
\begin{equation}\label{cancellation}[\gamma_{k-2},\gamma_{k-1}]\cap[\gamma_{k-1},\gamma_k]\cap[\gamma_k,\gamma_{k+1}]=\emptyset
\mbox{ for all }k\in \{2,\dots,n-1\}.\end{equation}
Then $[\gamma_0,\gamma_1]\cap [\gamma_{n-1},\gamma_n]=\emptyset$.
\end{Thm}
\begin{proof} First recall that in $\Gamma$, the intersection of any two connected subgraphs $\Lambda$ and $\Xi$ is connected \cite[Prop. 2.2]{geometry}: indeed, for any two vertices $u,v$ in $\Lambda\cap\Xi$, the geodesic $[u,v]$ is also contained in $\Lambda\cap \Xi$.

Suppose that the claim of the Theorem is not true and let $n\ge 3$ be minimal such that a counterexample  $\gamma_0,\dots, \gamma_n$ exists.
Set
$$\Delta=[\gamma_0,\gamma_1]\cap[\gamma_{n-1},\gamma_n],\quad\Theta=[\gamma_{n-2}\cap\gamma_{n-1}]\cap[\gamma_{n-1},\gamma_n]$$ and
$$\Lambda=[\gamma_0,\gamma_1]\cup[\gamma_1,\gamma_2]\cup\dots\cup[\gamma_{n-2},\gamma_{n-1}].$$
 Since $\gamma_0,\gamma_1,\dots,\gamma_n$ is a counterexample, $\Delta\ne\emptyset$.
Since $n$ is minimal for a counterexample,
$$\Lambda\cap[\gamma_{n-1},\gamma_n]=([\gamma_0,\gamma_1]\cup[\gamma_{n-2},\gamma_{n-1}])\cap[\gamma_{n-1},\gamma_n]=\Delta\cup \Theta.$$
Moreover, $\Delta\cap \Theta=\emptyset$. In case $n=3$ this is equivalent to
$$[\gamma_0,\gamma_1]\cap[\gamma_1,\gamma_2]\cap[\gamma_2,\gamma_3]=\emptyset;$$
for $n>3$ this follows from $[\gamma_0,\gamma_1]\cap[\gamma_{n-2},\gamma_{n-1}]=\emptyset$ (since $n$ is minimal for a counterexample). By Lemma \ref{disconnected}, the graph $$\Lambda\cap [\gamma_{n-1},\gamma_n]=\Delta\cup\Theta$$ is disconnected which contradicts the fact that  the intersection of two connected subgraphs of $\Gamma$ is connected.
\end{proof}

Note that  Theorem \ref{circuitforgraphs} immediately implies that in case (\ref{cancellation}) holds for $\gamma_0,\dots,\gamma_n$ then
$$[\gamma_{i-1},\gamma_i]\cap[\gamma_k,\gamma_{k+1}]=\emptyset$$
for all $1\le i < k\le n-1$. Theorem \ref{circuitforgraphs} is a generalization of Theorem 3.7 in \cite{geometry} but the proof presented here is much simpler than the one in \cite{geometry}.

 Now let $\mathcal G$ be an $A$-generated profinite group with tree-like Cayley graph $\Ga(\mathcal G)$; note that the abstract subgroup of $\mathcal G$ generated by $A$ is $F$.
\begin{Lemma}\label{empty} Let $n\ge 2$, $\ga_0,\dots,\ga_n\in \mathcal{G}$ and suppose that in case $n\ge 3$  for all $i\in\{2,\dots, n-1\}$ condition (\ref{cancellation}) holds. Let $w\in F$, $u$
be a prefix of $w$, that is $w=uv$ for some $v$ (equality of words) and suppose that $u\in [\ga_0,\ga_1]$ and $w\in [\ga_{n-1},\ga_n]$. Then $v$ admits a factorisation $v=v_1\cdots v_n$ such that, for $z_i=uv_1\cdots v_i$ (for $0\le i\le n$) then
$$z_i\in [\ga_{i-1},\ga_i]\cap[\ga_i,\ga_{i+1}]\mbox{ for } 1\le i \le n-1$$
and  $[z_{i-1},z_i]\subseteq[\ga_{i-1},\ga_i]$ for $1\le i\le n$.
\end{Lemma}
\begin{proof} The situation described in the lemma is depicted in Figure \ref{case2}.
\begin{figure}[ht]
\begin{tikzpicture}[scale=0.7]
\draw[thick] (0,1) to [out=90,in=180] (2,2) to [out=0,in=90] (4,1) to [out=-90,in=0]  (2,0) to [out=180,in=-90] (0,1);
\draw[thick] (2.5,1) to [out=90,in=180] (4.5,2) to [out=0,in=90] (6.5,1) to [out=-90,in=0]  (4.5,0) to [out=180,in=-90] (2.5,1);
\draw[thick] (5,1) to [out=90,in=180] (7,2) to [out=0,in=90] (9,1) to [out=-90,in=0]  (7,0) to [out=180,in=-90] (5,1);
\draw[dotted] (7.5,1) to [out=90,in=180] (9.5,2) to [out=0,in=90] (11.5,1) to [out=-90,in=0]  (9.5,0) to [out=180,in=-90] (7.5,1);
\draw[thick] (10,1) to [out=90,in=180] (12,2) to [out=0,in=90] (14,1) to [out=-90,in=0]  (12,0) to [out=180,in=-90] (10,1);
\draw[thick] (12.5,1) to [out=90,in=180] (14.5,2) to [out=0,in=90] (16.5,1) to [out=-90,in=0] (14.5,0) to [out=180,in=-90] (12.5,1);
\node at (1,1) {$\bullet$};
\node at (2,0.7) {$\bullet$};
\node[below] at (2,0.7) {$u$};
\node at (3.25,1.15) {$\bullet$};
\node at (3.25,0.75) {$\bullet$};
\node at (5.75,1.15) {$\bullet$};
\node at (5.75,0.75) {$\bullet$};
\node[below] at (5.75,0.75) {$z_2$};
\node at (8.25,1.15) {$\bullet$};
\node at (8.25,0.75) {$\bullet$};
\node[below] at (8.25,0.75) {$z_3$};
\node at (10.75,1) {$\bullet$};
\node at (10.75,0.3) {$\bullet$};
\node[below] at (10.75,0.2) {$z_{n-2}$};
\node at (13.25,1) {$\bullet$};
\node at (13.25,0.3) {$\bullet$};
\node[below] at (13.25,0.2) {$z_{n-1}$};
\node at (15.75,1) {$\bullet$};
\node at (14.75,0.7) {$\bullet$};
\node[below] at (14.75,0.7) {$w$};
\draw (2,0.7) -- (3.25,0.75) -- (5.75,0.75)--(8.25,0.75);
\draw[densely dotted] (8.25,0.75) -- (10.75,0.3);
\draw (10.75,0.3) -- (13.25,0.3) -- (14.75,0.7);
\node[above] at (1,1) {$\gamma_0$};
\node[above] at (3.25,1.15) {$\gamma_1$};
\node[above] at (5.75,1.15) {$\gamma_2$};
\node[above] at (8.25,1.15) {$\gamma_3$};
\node[above] at (10.75,1) {$\gamma_{n-2}$};
\node[above] at (13.25,1) {$\gamma_{n-1}$};
\node[above] at (15.75,1) {$\gamma_n$};
\node[below] at (3.25,0.75) {$z_1$};
\end{tikzpicture}
\caption{The graph $[\ga_0,\ga_1]\cup[\ga_1,\ga_2]\cup\cdots\cup[\ga_{n-1},\ga_n]$.}\label{case2}
\end{figure}
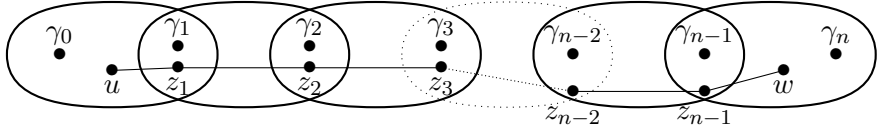
The proof is by induction on $n$. For $n=2$ we have $u \in [\ga_0,\ga_1]$, $uv=w\in [\ga_1,\ga_2]$. Since $[\ga_0,\ga_1]\cup[\ga_1,\ga_2]$ is connected, $[u,w]\subseteq [\ga_0,\ga_1]\cup[\ga_1,\ga_2]$ (since $\Gamma(\mathcal G)$ is a Hall graph). The geodesic $[u,w]$ has vertex $u$  in $[\ga_0,\ga_1]$ and $w$ in $[\ga_1,\ga_2]$ hence must go through a vertex in $[\ga_0,\ga_1]\cap[\ga_1,\ga_2]$. Let $v=v_1v_2$ be the corresponding factorization of the label $v$ of the path $u\to w$ then $uv_1\in [\ga_0,\ga_1]\cap[\ga_1,\ga_2]$. Since $u,uv_1\in [\ga_0,\ga_1]$ we also have $[u,uv_1]\subseteq [\ga_0,\ga_1]$, likewise $uv_1,uv_1v_2=w\in [\ga_1,\ga_2]$ whence $[uv_1,w]=[uv_1,uv_1v_2]\subseteq [\ga_1,\ga_2]$, as required.

Now suppose that $n\ge 3$ and the claim be true for $n-1$.
By assumption,
$$u,w\in [\ga_0,\ga_1]\cup[\ga_1,\ga_2]\cup\cdots\cup[\ga_{n-1},\ga_n]$$
whence $$[u,w]\subseteq [\ga_0,\ga_1]\cup[\ga_1,\ga_2]\cup\cdots\cup[\ga_{n-1},\ga_n]$$ since the latter is connected. Let
$$\De=[\ga_1,\ga_2]\cup\cdots\cup[\ga_{n-1},\ga_n].$$
Since $[u,w]\subseteq [\ga_0,\ga_1]\cup\De$, $u\in[\ga_0,\ga_1]$ and $w\in \De$ the geodesic $[u,w]$  must go through a vertex $x$ in $[\ga_0,\ga_1]\cap\De$. By Theorem \ref{circuitforgraphs} we have $[\ga_0,\ga_1]\cap[\ga_i,\ga_{i+1}]=\emptyset$ for all $i> 1$. Hence $x\in [\ga_0,\ga_1]\cap[\ga_1,\ga_2]$; then $u,x\in [\ga_0,\ga_1]$ implies $[u,x]\subseteq [\ga_0,\ga_1]$ (again since  $\Ga(\mathcal{G})$ is a Hall graph). Let $v_1$ be the label of the corresponding geodesic path $u\to x$ which is a prefix of $v$;   then, for $v=v_1z$, $z$ labels the geodesic path $x\to w$ which runs entirely in $\De$ since $x,w\in \De$ implies $[x,w]\subseteq \De$. By the induction hypothesis, $z$ admits a factorization $z=v_2\cdots v_n$ such that for $z_0:=u$, $z_1:=x=uv_1$ and $z_i=xv_2\dots v_i$ (for $2\le i\le n$) then $z_j\in [\ga_{j-1},\ga_j]\cap[\ga_j,\ga_{j+1}]$ for $1\le j\le n-1$ and $[z_{i-1},z_i]\subseteq[\ga_{i-1},\ga_i]$ for $i=1,\dots,n$.
\end{proof}

We have thus all prerequisites we need for a proof of the Ribes-Zalesski{\u\i}-Theorem for the pro-$\mathfrak{F}$-topology of a free group where $\mathfrak{F}$ is  arboreous.
\begin{Thm} Let $\fF$ be an arboreous formation, $n\ge 2$ and $H_1,\dots,H_n\le F$ be finitely generated $\fF$-extendible groups. Then the product $H_1\cdots H_n$ is pro-$\fF$-closed in $F$.
\end{Thm}

\begin{Rmk} Since an arboreous formation is \emph{a fortiori} Hall, being pro-$\fF$-closed and being $\fF$-extendible are equivalent for a finitely generated subgroup of $F$. Hence the formulation above is equivalent to the more familiar one: the product $H_1\cdots H_n$ is pro-$\fF$-closed for pro-$\fF$-closed subgroups $H_1,\dots, H_n$  of $F$.
\end{Rmk}

\begin{proof} We proceed by induction on $n$ with induction base $n=2$ being true by Proposition \ref{tlidc}. Let $n>2$ and suppose that the claim be true for $n'<n$. So let $H_1,\dots, H_n\le F$ be  $\fF$-extendible. By the induction hypothesis, for each $i\in \{1,\dots, n\}$ and $w\in F$ the set $H_1\cdots H_{i-1}wH_{i+1}\cdots H_n$ is pro-$\fF$-closed. Indeed, as mentioned above, being $\fF$-extendible is equivalent to being pro-$\fF$-closed which is a purely topological property and hence preserved by homeomorphisms of $F$. From the fact that the conjugation $x\mapsto wxw\inv$ is a homeomorphism it follows that each group ${}^wH_k=wH_kw\inv$ is $\fF$-extendible and hence, by induction hypothesis the set
$$H_1\cdots H_{i-1}w(H_{i+1}\cdots H_n)w\inv = H_1\cdots H_{i-1}{}^wH_{i+1}\cdots {}^wH_n$$
is pro-$\fF$-closed. Since right translation $x\mapsto xw$ is also a homeomorphism the claim follows.

Now take a sequence $(w_k)_{k\in \mathbb{N}}$ with $w_k\in H_1\cdots H_n$ such that $\lim w_k=w\in F$. We need to  show that $w\in H_1\cdots H_n$. Each $w_k$ can be written as
$$w_k=h_{1k}h_{2k}\cdots h_{nk}$$
with $h_{ik}\in H_i$ for all $i$ and $k$ and equality holds in $F$ (the product on the right hand side need not be reduced). We may assume (by going to a subsequence) that each of the sequences $(h_{ik})_{k\in \mathbb{N}}$ converges in $\wh{F_\fF}$. So, let $\eta_i=\lim_{k\to \infty}h_{ik}$ for $i=1,\dots,n$. In addition, we may assume that for all $i<j$, $\eta_i\cdots \eta_{j-1}\ne 1$ since otherwise we could replace each $w_k$ with
$$w'_k:=h_{1k}\cdots h_{i-1,k}h_{jk}\cdots h_{nk}.$$
Then $w=\lim w_k=\lim w'_k$ and
$$\lim w'_k\in H_1\cdots H_{i-1}H_j\cdots H_n\subseteq H_1\cdots H_n$$
by the induction hypothesis. Setting $\ga_0:=1$ and $\ga_i=\eta_1\cdots \eta_i$ for $i=1,\dots, n$ then we have that the elements $\ga_i$ are pairwise distinct. We distinguish two cases
\begin{enumerate}
\item there exists $i\in \{2,\dots,n-1\}$ such that
$$[\ga_{i-2},\ga_{i-1}]\cap[\ga_{i-1},\ga_i]\cap [\ga_i,\ga_{i+1}]\ne \emptyset,$$
\item for all $i\in \{2,\dots,n-1\}$,
$$[\ga_{i-2},\ga_{i-1}]\cap[\ga_{i-1},\ga_i]\cap [\ga_i,\ga_{i+1}]= \emptyset.$$
\end{enumerate}

Denote by $\cH_1,\dots,\cH_n$  the core-graphs of $H_1,\dots,H_n$ with base points $b_1,\dots, b_n$, completions $\ol{\cH_1},\dots,\ol{\cH_n}$ and transition groups $T_1,\dots, T_n\in \fF$, respectively. In case (1), let $i\in \{2,\dots,n-1\}$ be such that there exists
$$\xi\in [\ga_{i-2},\ga_{i-1}]\cap[\ga_{i-1},\ga_i]\cap[\ga_i,\ga_{i+1}].$$
Take an $A$-generated group $G_0\in \fF$ such that $G_0\twoheadrightarrow T_l$ for all $l$ and such that all elements $[\ga_l]_{G_0}$ are pairwise distinct. Choose an inverse system $(G_m)$ of $A$-generated groups in $\fF$ such that
$G_0\twoheadleftarrow G_1\twoheadleftarrow G_2\twoheadleftarrow\cdots$ and $\vpl G_m=\wh{F_{\fF}}$ and denote the canonical morphism $\wh{F_{\fF}}\twoheadrightarrow G_m$ by $\vp_m$. We may assume, by considering an appropriate subsequence $(w_{k_m})$ of $(w_k)$, that
$$[h_{lm}]_{G_m}=[\eta_l]_{G_m}$$ for all $l=1,\dots, n$ and all $m\in \mathbb{N}$.

For $\xi\in [\ga_{i-2},\ga_{i-1}]\cap[\ga_{i-1},\ga_i]\cap[\ga_i,\ga_{i+1}]$ we note that
$$[\ga_{i-1},\xi]\subseteq [\ga_{i-2},\ga_{i-1}]\cap[\ga_{i-1},\ga_i]\mbox {
and }[\xi,\ga_i]\subseteq [\ga_{i-1},\ga_i]\cap[\ga_i,\ga_{i+1}].$$
Choose $m\in\mathbb{N}$ and consider the image of
$$[\ga_{i-2},\ga_{i-1}]\cup[\ga_{i-1},\ga_i]\cup [\ga_i,\ga_{i+1}]$$
under $\vp_m$ (that is, within the graph $\Ga(G_m)$) and set $g_k:=\ga_k\vp_m$ for $k=i-2,i-1,i,i+1$ and $x=\xi\vp_m$. The words $h_{i-1,m}, h_{im}, h_{i+1,m}$ label paths $g_{i-2}\to g_{i-1}$, $g_{i-1}\to g_i$ and $g_i\to g_{i+1}$,  all in $\Ga(G_m)$, but not necessarily inside $([\ga_{i-2},\ga_{i-1}]\cup[\ga_{i-1},\ga_i]\cup [\ga_i,\ga_{i+1}])\vp_m$. We choose words $e_m$, $f_m$ such that $e_m$ labels a path $p_m:g_{i-1}\to x$ which runs inside $[\ga_{i-2},\ga_{i-1}]\vp_m\cap [\ga_{i-1},\ga_i]\vp_m$, and $f_m$ labels a path $q_m:x\to g_i$ which runs inside $[\ga_{i-1},\ga_{i}]\vp_m\cap [\ga_{i},\ga_{i+1}]\vp_m$.
The situation is depicted in Figure \ref{case1Gm}.
\begin{figure}[ht]
\begin{tikzpicture}[xscale=0.75]
\draw[thick] (0,0) to [out=90,in=180] (4.5,1) to [out=0,in=90] (8.8,0) to [out=-90,in=0]  (4.5,-1) to [out=180,in=-90] (0,0);
\draw[thick] (7.3,0) to [out=90,in=180] (11.5,1) to [out=0,in=90] (16,0) to [out=-90,in=0]  (11.5,-1) to [out=180,in=-90] (7.3,0);
\node at (1,0) {$\bullet$};
\node at (6,0) {$\bullet$};
\node at (10,0) {$\bullet$};
\node at (15,0) {$\bullet$};
\node at (8,0) {$\bullet$};
\node[above] at (1,0) {$g_{i-2}$};
\node[above right] at (6,-0.1) {$g_{i-1}$};
\node[above right] at (10,-0.1) {$g_i$};
\node[above] at (15,0) {$g_{i+1}$};
\node[above] at (8,0) {$x$};
\draw[thick] (5,0) to [out=90,in=180] (8,1) to [out=0,in=90] (11,0) to [out=-90,in=0]  (8,-1) to [out=180,in=-90] (5,0);
\draw[->,decorate, decoration={snake}] (1,0) to  [out=-60,in=-120] (5.95,-0.05);
\draw[->,decorate, decoration={snake}] (10,0) to  [out=-60,in=-120] (14.95,-0.05);
\draw[->,decorate, decoration={snake}] (6,0) to  [out=80,in=100] (10,0.08);
\draw[->] (6,0) to [out=-20,in=-160] (7.95,-0.05);
\draw[->] (8,0)to [out=-20,in=-160] (9.95,-0.05);
\node[below] at (3.4,-1.27) {$h_{i-1,m}$};
\node[below] at (12.4,-1.27) {$h_{i+1,m}$};
\node[above] at (8,1.25) {$h_{im}$};
\node[below] at (7,-0.15) {$e_m$};
\node[below] at (9.2,-0.15) {$f_m$};
\end{tikzpicture}
\caption{The graph $([\ga_{i-2},\ga_{i-1}]\cup[\ga_{i-1},\ga_i]\cup[\ga_i,\ga_{i+1}])\vp_m$.}\label{case1Gm}
\end{figure}

Consequently,
$[h_{im}]_{G_m}=[e_mf_m]_{G_m}.$
Doing so for every $m\in \mathbb{N}$ we see that
$$\lim_{m\to\infty}e_mf_m=\lim_{m\to\infty}h_{im}=\eta_i.$$
It follows that
$$\lim_{m\to\infty}h_{1m}\cdots h_{i-1,m}e_mf_mh_{i+1,m}\cdots h_{nm}=\lim_{m\to\infty} w_m=w.$$
Now observe that, for each $m$, the canonical projection $g_{i-2}\cH_{i-1}^{G_m}\twoheadrightarrow\cH_{i-1}$ maps $g_{i-2}$ as well as $g_{i-1}$ to the base point $b_{i-1}$ and hence the path $p_m$ to a path $p'_m$ in $\cH_{i-1}$ starting at the base point $b_{i-1}$. For every $m$, $p'_m$ also ends at the same vertex, say $v$, namely at the image of $\xi$ under the canonical mapping $\ga_{i-2}\cH_{i-1}^{\wh{F_{\fF}}}\twoheadrightarrow \cH_{i-1}$.
Let $e$ be a word labelling a fixed path in $\cH_{i-1}$ from the base point $b_{i-1}$ to  $v$. Then $e_me\inv$ labels a closed path at $b_{i-1}$ in $\cH_{i-1}$, whence $e_me\inv \in H_{i-1}$ for each $m$. Likewise, there exists a word $f$ labelling a path inside $\cH_{i+1}$ from some fixed vertex $u$ to the base point $b_{i+1}$ and  $f\inv f_m\in H_{i+1}$ for each $m\in \mathbb{N}$. Consequently, for every $m$, the word
$$w_m'=h_{1m}\cdots (h_{i-1,m}e_me\inv)ef(f\inv f_mh_{i+1,m})\cdots h_{nm}$$
belongs to $H_1\cdots H_{i-1}efH_{i+1}\cdots H_n$ and
$$\lim_{m\to \infty}w_m'=\lim_{m\to\infty}w_m=w.$$
By the induction hypothesis, $w\in H_1\cdots H_{i-1}efH_{i+1}\cdots H_n$. Hence there exist $h_k\in H_k$ (for $k=1,\dots,i-1,i+1,\dots,n$) such that
$$w=h_1\cdots h_{i-1}efh_{i+1}\cdots h_n.$$
Finally, let us look once more at the word $e_mf_m$ for some fixed $m$, say $m=0$. We know that $e_0f_0$ labels a path $\ga_{i-1}\vp_0\to \ga_i\vp_0$ which runs entirely in $[\ga_{i-1},\ga_i]\vp_0$. Under the canonical mapping $(\ga_{i-1}\vp_0)\cH_i^{G_0}\twoheadrightarrow \cH_i$ that path is mapped to a closed path at base point $b_i$. Consequently, $e_0f_0\in H_i$. Since, by construction, $ee_0\inv \in H_{i-1}$ and $f_0\inv f \in H_{i+1}$ we arrive at
$$w=h_1h_2\cdots\underbrace{h_{i-1}ee_0\inv}_{\in H_{i-1}}\underbrace{e_0f_0}_{\in H_i} \underbrace{f_0\inv fh_{i+1}}_{\in H_{i+1}}\cdots h_n,$$
thus $w\in H_1\cdots H_n$, as required.

Now consider case (2). We have $\ga_0=1$ and $\ga_n=w$. By Lemma \ref{empty}, $w$ admits a factorization $w=v_1v_2\cdots v_n$ such that, for $z_i=v_1\cdots v_i$ (for $0\le i\le n$) then
$$z_i\in [\ga_{i-1},\ga_i]\cap[\ga_i,\ga_{i+1}]\mbox{ for }1\le i\le n-1$$
and for $1\le i\le n$, $[z_{i-1}, z_i]\subseteq[\ga_{i-1},\ga_i]$. Now choose an $A$-generated group $G\in \fF$ with canonical morphism $\vp:\wh{F_{\fF}}\twoheadrightarrow G$ such that
\begin{itemize}
\item $G\twoheadrightarrow T_i$ for all $i$,
\item all $\ga_i\vp$ are pairwise distinct,
\item $\vp$ restricted to $[1,w]$ is injective.
\end{itemize}
We consider the graph
$$([\ga_0,\ga_1]\cup[\ga_1,\ga_2]\cup\cdots\cup[\ga_{n-1},\ga_n])\vp$$
in $\Ga(G)$. For each $i=1,\dots,n-1$,
$$[z_i,\ga_i]\vp\subseteq [\ga_{i-1},\ga_i]\vp\cap[\ga_i,\ga_{i+1}]\vp;$$
let $s_i$ be a word labelling a path $z_i\vp\to\ga_i\vp$ running inside $[z_i,\ga_i]\vp$, and set $s_0=1=s_n$. Then: for each $i=1,\dots,n$ the word $s_{i-1}\inv v_is_i$ labels a path $p_i$ $\ga_{i-1}\vp\to \ga_i\vp$ running entirely inside $[\ga_{i-1},\ga_i]\vp$. Similarly as argued earlier, the canonical projection $(\ga_{i-1}\vp)\cH_i^G\twoheadrightarrow \cH_i$ maps $\ga_{i-1}\vp$ as well as $\ga_i\vp$ to the base point $b_i$. Since $[\ga_{i-1},\ga_i]\varphi\subseteq (\ga_{i-1}\varphi)\cH^G$, the path $p_i$ is thereby mapped to a closed path at base point $b_i$. Consequently, $s_{i-1}\inv v_is_i\in H_i$ for $i=1,\dots, n$. Altogether,
$$w=v_1\cdots v_n=v_1s_1\cdot s_1\inv v_2s_2\cdots s_{n-2}\inv v_{n-1}s_{n-1}\cdot s_{n-1}\inv v_n\in H_1\cdots H_n,$$
as required.
\end{proof}

\section{Constructing  groups with tree-like Cayley graphs}\label{S-extensions}

We present a method to construct inverse sequences $G_0\twoheadleftarrow G_1\twoheadleftarrow G_2\twoheadleftarrow\cdots$ for which the group $\mathcal{G}=\vpl G_n$ has a tree-like Cayley graph. Suppose that $\mathcal{G}$ is an $A$-generated profinite group whose Cayley graph admits a constellation $(\Xi,\ga,\Th)$. Then there exists a finite quotient $G$ with canonical morphism $\vp:\mathcal{G}\twoheadrightarrow G$ such that $(\Xi\vp,\ga\vp,\Th\vp)$ is a constellation in $\Ga(G)$. Moreover, if $\vp$ factors through another finite quotient $H$, say: $\mathcal{G}\overset{\psi}{\twoheadrightarrow}H\overset{\mu}{\twoheadrightarrow}G$, then $(\Xi\psi,\ga\psi,\Th\psi)$ is  a constellation in $\Ga(H)$ and
$$(\Xi\psi\mu,\ga\psi\mu,\Th\psi\mu)=(\Xi\vp,\ga\vp,\Th\vp).$$
In particular, there exist words $u,v\in F$ such that $u$ labels a path $1\to \ga\vp$ running in $\Xi\vp$, $v$ labels a path $1\to \ga\vp$ in $\Th\vp$ and $[u]_H=[v]_H$: simply take words $u,v$ which label paths  $1\to \ga\psi$ in $\Ga(H)$ such that the one labelled by $u$ runs in $\Xi\psi$ and the one labelled by $v$ runs in $\Th\psi$.

Now we come to a crucial definition. Let $(X,g,T)$ be a constellation in the Cayley graph $\Ga(G)$ of the finite $A$-generated group $G$; let $H$ be another $A$-generated group such that $H\twoheadrightarrow G$. Then $H$ \emph{dissolves} the constellation $(X,g,T)$ if, for all pairs of words $(u,v)$ such that $u$ labels a path $1\to g$ running in $X$ and $v$ labels a path $1\to g$ running in $T$, the inequality $[u]_H\ne[v]_H$ holds. If $H$ dissolves the constellation $(X,g,T)$ of $G$ then, in particular,  there does not exist a constellation $(Y,h,Z)$ of $H$ for which $(X,g,T)=(Y\vp,h\vp,Z\vp)$ for the canonical morphism $\vp:H\twoheadrightarrow G$. Moreover, each $A$-generated group $K$ for which $K\twoheadrightarrow H$ then also dissolves the constellation $(X,g,T)$ of $G$. This, together with the earlier discussion implies that $\Ga(\mathcal G)$ is tree-like provided that for each finite quotient $G$ of $\mathcal G$ and each constellation $(X,g,T)$ of $G$  there exists another quotient $H$ of $\mathcal{G}$ which dissolves the constellation $(X,g,T)$. Now, every finite group can have at most finitely many constellations $(X,g,T)$. If, for every such $(X,g,T)$ we have a group $H_{(X,g,T)}$ which dissolves $(X,g,T)$ then, taking the $A$-generated subdirect product $H$ of all such $H_{(X,g,T)}$, we get a group which dissolves all constellations of $G$. Thus, we have already proved the `if'-direction of  the next result.

\begin{Prop}\label{treelike} The Cayley graph $\Gamma(\mathcal G)$ of an $A$-generated profinite group $\mathcal G$ is tree-like if and only if, for each finite quotient $G$ of $\mathcal G$ there exists a finite quotient $H$ of $\mathcal G$ which dissolves all constellations of $G$.
\end{Prop}
\begin{proof} We only have to prove the `only if'-direction. Suppose that $\mathcal G$ does not fulfill the condition stated in the proposition. Then there exists a finite quotient $G$ of $\mathcal G$ such that for every $H$ with $\mathcal{G}\twoheadrightarrow H\twoheadrightarrow G$ there exists a constellation $(X,g,T)$ of $G$ which is not dissolved by $H$. Let $G=:G_0\twoheadleftarrow G_1\twoheadleftarrow G_2\twoheadleftarrow\cdots$ be an inverse system of quotients of $\mathcal G$ such that $\vpl G_n=\mathcal G$. For every $n$ there exists a constellation $(X_n,g_n,T_n)$ of $G$ which is not dissolved by $G_n$. Since $G$ has only finitely many constellations we may assume that  the constellations $(X_n,g_n,T_n)$ all coincide with a fixed one, say $(X,g,T)$. By definition, for each $n$, there exist words $u_n$ and $v_n$ such that $[u_n]_{G_n}=[v_n]_{G_n}$, $u_n$ labels a path $1\to g$ inside $X$ and $v_n$ labels a path $1\to g$ inside $T$. We may assume (by going to subsequences) that both sequences $([u_n]_{\mathcal G})$ and $([v_n]_{\mathcal G})$ converge, in which case they have the same limit, say $\gamma:=\lim\, [u_n]_{\mathcal G}=\lim\, [v_n]_{\mathcal G}$.  Now consider the covering graphs $X^{\mathcal G}$ and $T^{\mathcal G}$ (both with respect to the base point $1$). Both graphs are connected subgraphs of $\Gamma(\mathcal G)$ and both contain $\gamma$ and $1$. Since the canonical mapping $\Gamma(\mathcal G)\twoheadrightarrow \Gamma(G)$ maps $X^{\mathcal G}$ to $X$, $T^{\mathcal G}$ to $T$, $\gamma$ to $g$ and $1$ to $1$ it follows that $(X^{\mathcal G},\gamma,T^{\mathcal G})$ is a constellation in $\Gamma(\mathcal G)$ whence $\Gamma(\mathcal G)$ is not tree-like.
\end{proof}
We continue by a technical lemma saying that in Proposition \ref{treelike} it is not necessary to consider \textbf{all} finite quotients of $\mathcal G$ but rather a co-final set.

\begin{Lemma}\label{co-final} Let $G\twoheadrightarrow H\overset{\vp}{\twoheadrightarrow}K$ be finite, $A$-generated. If $G$ dissolves all constellations of $H$ then $G$ dissolves all constellations of $K$.
\end{Lemma}
\begin{proof} Let $(X,g,T)$ be a constellation of $K$ and let $u, v$ be words labelling paths $1\to g$ in $X$ and $T$, respectively. If $[u]_H\ne [v]_H$ then also $[u]_G\ne [v]_G$ and we are done. So, assume that $[u]_H=[v]_H=:h$. Let $Y$ be the subgraph of $\Ga(H)$ spanned by the edges of the path $u:1\to h$ and $Z$ be the subgraph of $\Ga(H)$ spanned by the edges of the path $v:1\to h$ in $\Ga(H)$. Then $Y\vp\subseteq X$, $Z\vp\subseteq T$, $h\vp=g$ and therefore $(Y,h,Z)$ is a constellation in $H$. Since $G$ dissolves $(Y,h,Z)$ it follows that $[u]_G\ne[v]_G$, as required.
\end{proof}

We need some further preparations.

\begin{Lemma}\label{FSab} Let $S$ be a finite, simple group and $F_{\mathfrak{for}(S)}(a,b)$ be the $2$-generator free object in the formation $\mathfrak{for}(S)$ generated by $S$. Then $a^mb^n= 1$ only if $a^m=1=b^n$.
\end{Lemma}
\begin{proof} If $S=C_p$ for some prime $p$ this is clear since $F_\mathfrak{for}(S)(a,b)$ is the (additive group of) the $\mathbb{F}_p$-vector space with basis $\{a,b\}$. So assume that $S$ is non-abelian and let $x\in \left<a\right>\cap\left<b\right>$. Then $x$ commutes with $a$ as well as with $b$ and hence belongs to the center of $F_{\mathfrak{for}(S)}(a,b)$ since $a$ and $b$ generate $F_{\mathfrak{for}(S)}(a,b)$. However, since $F_{\mathfrak{for}(S)}(a,b)$ is a direct power of $S$ this group is centerless, so $x=1$ and $\left<a\right>\cap\left<b\right>=\{1\}$.
\end{proof}

For a finitely generated free group $R$ and a finite simple group $S$ let $R(S)$ be the intersection of all normal subgroups $N$ of $R$ for which $R/N$ is a direct power of $S$. Then $R(S)$ is a characteristic subgroup of $R$ and $R/R(S)$ is the $r$-generator free object in the formation $\mathfrak{for}(S)$ generated by $S$ where $r$ is the rank of $R$. In case $S=C_p$ for a prime $p$ we have $R(S)=R(C_p)=R^p[R,R]$. If $w_1,\dots,w_r$ form a basis of $R$ then $w_1R(S),\dots,w_rR(S)$ form a collection of free generators of $R/R(S)$.

Let $G$ be an $A$-generated finite group, $\vp:F\to G$ be the canonical morphism and $R=\mathrm{ker}(\vp)$. We define $G^{A,S}:= F/R(S)$ and call this group the \emph{$A$-universal $S$-extension of $G$}. For $S=C_p$ this is a classical construction by Gasch\"utz, see \cite[Appendix $\beta$]{doerkhawkes} and \cite{Ballester}. Note that $G^{A,S}$ is an $A$-generated extension of $R/R(S)$ by $G$ and $R/R(S)$ is the $r$-generator free object in the formation generated by $S$ where $r=\vert G\vert(\vert A\vert - 1)+1$; $G^{A,S}$ depends on $S$ \textbf{and} $A$ and enjoys the following universal property: if $H$ is any $A$-generated extension of a member of $\mathfrak{for}(S)$ by $G$ then
$$G^{A,S}\twoheadrightarrow H\twoheadrightarrow G.$$

So we are interested, given a finite $A$-generated group $G$, to find a group $H\twoheadrightarrow G$ which dissolves all constellations of $G$. Since our ultimate goal is to give a sufficient condition for $\wh{F_{\fF}}$ to be tree-like with respect to a free generating set $A$, in view of Lemma \ref{co-final}, it is no harm if we impose a few mild restrictions. Since the cyclic and pro-cyclic case is already clear (every infinite pro-cyclic group has a tree-like Cayley graph \cite{geometry}) we generally assume that $\vert A\vert \ge 2$. Since, among the finite quotients of $\wh{F_{\fF}}$ the relatively free ones form a co-final subset we shall assume that
\begin{equation}\label{assumption}
[a]_G\ne [b]_G\ne 1\mbox{ for all }a,b\in A, a\ne b.
\end{equation}
\begin{Lemma}\label{connected-2} Let $G$ be an $A$-generated finite group subject to (\ref{assumption}) and let $e,f$ be arbitrary positive edges in the Cayley graph $\Ga(G)$; then the graph $\Ga(G)\setminus\{e^{\pm 1},f^{\pm 1}\}$ is connected.
\end{Lemma}
\begin{proof} The claim is true if $G$ is generated by two elements each of order $2$  (that is, $G$ is a dihedral group). So we may assume that $\vert A\vert \ge 3$ or $\vert A\vert =2$ and at least one generator has order greater than $2$. Let $\Ga^\circ$ be the undirected graph formed from the positive edges of $\Ga(G)$ by ignoring the orientation of these edges. Moreover, from each pair of edges $g\ne g'$ coming from a generator of order $2$ (that is, $\iota g=\tau g'$ and $\iota g'=\tau g$) remove one in order to get a graph rather than a multi-graph. The graph $\Ga^\circ$ is vertex transitive with degree at least $3$. It follows that the edge connectivity of $\Ga^\circ$ is also at least $3$ \cite[Lemma 3.3.3]{graphtheory}. That is, we may remove any two edges from $\Ga^\circ$ and retain a connected graph. It follows that $\Ga(G)\setminus \{e^{\pm 1},f^{\pm 1}\}$ is also connected.
\end{proof}

We are now able to prove the main result of this section. For $S=C_p$ for some prime $p$ this result is known \cite{geometry}.

\begin{Thm} Let $G$ be an $A$-generated finite group subject to the assumption (\ref{assumption}) and let $S$ be a finite simple group. Then $G^{A,S}$ dissolves all constellations of $G$.
\end{Thm}
\begin{proof} Let $(X,g,T)$ be a constellation of $G$ and $u,v\in F$ be words such that $[u]_G=g=[v]_G$ and $u$ labels a path $1\to g$ inside $X$, $v$ labels a path $1\to g$ inside $T$. We need to show that $[u]_{G^{A,S}}\ne[v]_{G^{A,S}}$, that is, $[uv\inv]_{G^{A,S}}\ne 1$.

The group $G^{A,S}$ is an extension of $R/R(S)$ by $G$ where $R$ is the kernel of the canonical morphism $F\twoheadrightarrow G$, that is, $R$ is an absolutely free group of rank $r=\vert G\vert(\vert A\vert -1)+1$ and $R/R(S)$ is a relatively free group of rank $r$ in the formation $\mathfrak{for}(S)$ generated by $S$. From $[uv\inv]_G=1$ it follows that $uv\inv \in R$; moreover $R$ is a finitely generated subgroup of $F$ with core-graph $\cR=\Ga(G)$ with base point $1$. We actually need to show that $[uv\inv]_{R/R(S)}\ne 1$.

We know how to construct a basis of $R$: for every spanning tree $\Upsilon$ of $\Ga(G)$ there exists a basis $B_\Upsilon$ whose elements are in bijective correspondence to the positive edges of $\Ga(G)\setminus \Upsilon$. We are going to select a tree which is suitable for our purpose. We shall argue in the same way as in \cite{mytype2,constructive}. For a path $\pi$ in $\Ga(G)$ and a positive edge $e$ denote by $\pi(e)$ the number of signed traversals of $e$ by $\pi$ (that is, whenever $\pi$ traverses $e$ in the forward direction this counts $+1$, in the backward direction $-1$).

Let $Z$ be the connected $X\cap T$-component containing $1$ and let $D$ be the set of all positive edges in $X$ with initial vertex in $Z$ and terminal vertex not in $Z$; likewise, let $C$ be the set of positive edges of $X$ with initial vertex not in $Z$ and terminal vertex in $Z$. As in \cite{mytype2}, $D\cup C\ne \emptyset$ and the edges of $D\cup C$ form a border that must be traversed by any path $\pi:1\to g$ inside $X$ one times more often in the forward direction than in the backward direction. That is, for any such path $\pi$,
\begin{equation}\label{borderX}
\sum_{e\in D}\pi(e)-\sum_{e\in C}\pi(e)=1.
\end{equation}
Analogously, let $D'$ be the set of positive edges in $T$ with initial vertex in $Z$ and terminal vertex in $T\setminus Z$, and $C'$ be the set of positive edges in $T$ with terminal vertex in $Z$ and initial vertex in $T\setminus Z$. For every path $\pi:1\to g$ inside $T$,
\begin{equation}\label{borderT}
\sum_{e\in D'}\pi(e)-\sum_{e\in C'}\pi(e)=1.
\end{equation}
The situation is depicted in Figure \ref{constellation}.
\begin{figure}[ht]
\begin{tikzpicture}[yscale=0.7]
\draw[thick] (0,0) -- (1,1) -- (2,0) -- (1,-1) -- (0,0);
\draw[thick] (6,0) -- (7,1) -- (8,0) -- (7,-1) -- (6,0);
\draw[thick] (1,1) -- (1.5,1.5); \draw[thick] (2,0) -- (2.5,0.5);
\draw[thick] (2,0) -- (2.5,-0.5);
\draw[thick] (1,-1) -- (1.5,-1.5);
\draw[thick] (1.5,1.5) to [out=45, in=135] (7,1);
\draw[thick] (2.5,0.5) to [out=45, in=135] (6,0);
\draw[thick] (2.5,-0.5) to [out=-45, in=-135] (6,0);
\draw[thick] (1.5,-1.5) to [out=-45, in = -135] (7,-1);
\draw (1.2,0.8)--(1.7,1.3);
\draw (1.4,0.6)--(1.9,1.1);
\draw (1.6,0.4)--(2.1,0.9);
\draw (1.8,0.2)--(2.3,0.7);
\draw (1.2,-0.8)--(1.7,-1.3);
\draw (1.4,-0.6)--(1.9,-1.1);
\draw (1.6,-0.4)--(2.1,-0.9);
\draw (1.8,-0.2)--(2.3,-0.7);
\node at (0.8,0) {$\bullet$};
\node at (7,0) {$\bullet$};
\node at (4,2) {$X$};
\node at (4,-2) {$T$};
\node[below] at (0.8,0) {$1$};
\node[below] at (7,0) {$g$};
\node[right,rotate =-35,scale=0.9] at (1.7,1.7) {$D\cup C$};
\node[right,rotate=35,scale=0.86] at (1.63,-1.7) {$D'\cup C'$};
\node at (1.5,0) {$Z$};
\end{tikzpicture}
\caption{Constellation $(X,g,T)$, borders $D\cup C$ and $D'\cup C'$.}\label{constellation}
\end{figure}
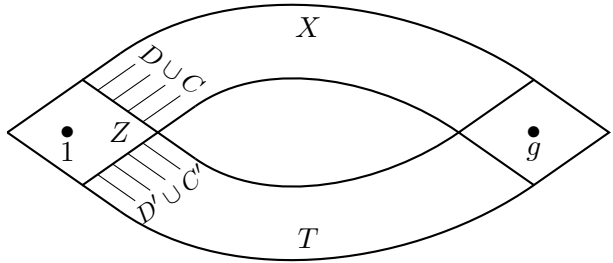

From the definition it is immediately clear that $(D\cup C)\cap (D'\cup C')=\emptyset$. Now let $o$ be the order of a free generating element of $F_{\mathfrak{for}(S)}(a,b)$, the $2$-generator free object in $\mathfrak{for}(S)$, and let $\wh u$ and $\wh v$ be the paths $1\to g$ in $X$ and $T$, labelled $u$ and $v$, respectively. By (\ref{borderX}, \ref{borderT}) there exist $e\in D\cup C$ and $f\in D'\cup C'$ such that $\wh u(e)$ and $\wh v(f)$ both are not divisible by $o$, and $e\ne f$ since $(D\cup C)\cap(D'\cup C')=\emptyset$. By Lemma \ref{connected-2} the graph $\Ga(G)\setminus\{e^{\pm 1},f^{\pm 1}\}$ is connected, so we my choose a spanning tree $\Upsilon$ of $\Gamma(G)$ which does not contain $e$ and $f$. For a positive edge $h\in \Ga(G)\setminus \Upsilon$  let $\til h$ be the label of the reduced path obtained by running inside $\Upsilon$ from $1$ to $\iota h$, traversing $h$ (in the forward direction) and then running back from $\tau h$ to $1$ inside $\Upsilon$. The basis $B_\Upsilon$ then consists of all $\til h$, $h$ a positive edge  in $\Gamma(G)\setminus\Upsilon$. The label of every closed path at $1$ in $\Ga(G)$  can be expressed uniquely as a reduced product of elements of $B_\Upsilon\cup B_\Upsilon\inv$.  For a given closed path $\pi=e_1e_2\dots e_k$ we know how to express its label as such a product: simply replace each positive edge $e_i$ by $\til{e_i}$ provided $e_i\notin \Upsilon$ and by $1$ if $e_i\in \Upsilon$; in case $e_i$ is a negative edge replace it by $\til{e_i\inv}\inv$ if $e_i\inv\notin\Upsilon$ and by $1$ if $e_i\inv \in \Upsilon$.
 Doing so for the path labelled $uv\inv$ we observe that $uv\inv =w(u)\cdot w(v\inv)$ where
\begin{itemize}
\item $w(u)$ and $w(v\inv)$ both are products of members of $B_\Upsilon\cup B_\Upsilon\inv$,
\item $w(u)$ does not contain $\til{ f}^{\pm 1}$ and $w(v\inv)$ does not contain $\til{ e}^{\pm 1}$.
\end{itemize}
Setting
\begin{itemize}
\item $u(e):=$ the sum of exponents of $\til e$ in $w(u)$,
\item $-v(f):=$ the sum of exponents of $\til f$ in $w(v\inv)$
\end{itemize}
we have that $u(e)=\wh u(e)$ and $v(f)=\wh v(f)$ and therefore $ o\mathrel{\not{\mid}} u(e)$ and $o\mathrel{\not{\mid}} v(f)$.
Consider now   the morphism $\vp:R\to F_{\mathfrak{for}(S)}(a,b)$ determined by the map $\til e\mapsto a$, $\til f\mapsto b$, $\til h\mapsto 1$ for $h\ne e,f$. Then
$$(uv\inv)\vp=(w(u))\vp\cdot (w(v\inv))\vp=a^{u(e)}b^{-v(f)}\ne 1$$
by Lemma \ref{FSab} since $o\mathrel{\not{\mid}} u(e)$ and $o\mathrel{\not{\mid}} v(f)$. Since $\vp$ factors through $R/R(S)$ we get
$$[uv\inv]_{R/R(S)}\psi=(uv\inv)\vp=a^{u(e)}b^{-v(f)}\ne 1$$
where $\psi:R/R(S)\to F_{\mathfrak{for}(S)}(a,b)$ is the morphism determined by $${\til e}R(S)\mapsto a,\ {\til f}R(S)\mapsto b,\ {\til h}R(S)\mapsto 1\ (h\ne e,f).$$ It follows that $[uv\inv]_{R/R(S)}\ne 1$ whence $[uv\inv]_{G^{A,S}}\ne 1$.
\end{proof}

\begin{Cor}\label{sufficient} Let $\mathcal G=\vpl G_i$ be an $A$-generated profinite group; if for each $i$ there exists a simple group $S_i$ with $\mathcal{G}\twoheadrightarrow G_i^{A,S_i}$ then $\Ga(\mathcal{G})$ is tree-like.
\end{Cor}
For a prime $p$, there is the notion of \emph{pro-$p$-tree} which is homologically defined \cite{Mel, RZ3}. From \cite[Theorem 3.9]{AWsurvey}, see also \cite[Theorem 9.6]{geometry}, one gets that $\Gamma(\mathcal G)$ is a pro-$p$-tree if and only if $\mathcal{G}\twoheadrightarrow G_i^{A,C_p}$ for \textbf{every} $i$.

We are motivated to call a formation $\fF$  \emph{locally extensible} if for every $A$-generated member $G$ of $\fF$ there exists a simple group $S$ such that $G^{A,S}$ also belongs to $\fF$. This seems to be the adequate analogue for formations of a notion introduced for varieties in \cite{geometry}. All formations closed under taking normal subgroups and extensions (called $NE$-formations in  \cite{RZbook1}) are locally extensible, but the converse is not true.
\begin{Cor} Every locally extensible formation $\fF$ is arboreous and therefore Hall.
\end{Cor}
The author is not aware of an example of an arboreous formation which is not locally extensible. The problems of whether there exist Hall varieties which are not arboreous and arboreous varieties which are not locally extensible have been formulated in \cite{geometry,power,supersolvable}. The problem to classify all Hall formations has been proposed in \cite{ballpinxaro}.
\begin{Problem} Does there exist an arboreous formation (or variety) which is not locally extensible?
\end{Problem}

Finally, let us consider an example. For a simple non-abelian group $S$ (and $\vert A\vert \ge 2$) let $R_1:=F(S)$ and  $R_n=R_{n-1}(S)$ for $n>1$. Let $\mathfrak{S}$ be the formation consisting of all finite groups all of whose composition factors are isomorphic with $S$. Then $\wh{F_{\fS}}= \vpl F/R_n$. By Corollary \ref{sufficient},  the Cayley graph $\Ga(\wh{F_{\fS}})$ is tree-like and $\fS$ is arboreous. However, $\wh{F_{\fS}}$ is perfect. Let $a\in A$; then the closed normal subgroup of $\wh{F_{\fS}}$ generated by $A\setminus\{a\}$ is the whole group $\wh{F_{\fS}}$ since otherwise it would have a non-trivial abelian quotient. It follows that the answer to Problem 7.8 in \cite{geometry} is ``No'', so Problem 7.7 (in the same paper) cannot be attacked by the approach proposed there.
\vskip0.05cm
\noindent\textbf{Acknowledgement}: The author would like to thank A.~Ballester-Bolinches for inviting him to the Department of Algebra at the University of Valencia; he is grateful for having had stimulating discussions with him  and J.~Cossey.

\end{document}